\documentclass{article}
\usepackage[english]{babel}
\usepackage{amsmath}
\usepackage{amsthm}
\usepackage{amssymb}
\usepackage{amscd}
\usepackage[latin1]{inputenc}
\usepackage[all]{xy}

\newtheorem{teor}{Theorem}[section]
\newtheorem{defi}{Definition}
\newtheorem{lemma}[teor]{Lemma}
\newtheorem{prop}[teor]{Proposition}
\newtheorem{cor}[teor]{Corollary}
\newtheorem{rem}[teor]{Remark}
\newtheorem{rems}[teor]{Remarks}
\newtheorem{exem}[teor]{Example}
\newtheorem{exems}[teor]{Examples}
\newtheorem{ques}[teor]{Question}

\topskip=\baselineskip  \headsep=0.3in \headheight=0in
\newcommand{\pullbackcorner}[1][dr]{\save*!/#1+1.2pc/#1:(1,-1)@^{*}\restore}
\newcommand{\pushoutcorner}[1][dr]{\save*!/#1-1.2pc/#1:(-1,1)@^{*}\restore}

\title{Direct limits in the heart of a t-structure: the case of a torsion pair}
\author{Carlos E. Parra  \thanks{The authors thank Silvana Bazzoni and Jan Trlifaj for their quick answer to our questions, and
Riccardo Colpi for  sending us the manuscript \cite{CG}.}
\\
Departamento de Matem\'aticas\\ Universidad de los Andes \\ ({\bf 5101}) M\'erida\\ VENEZUELA\\
{\it carlosparra@ula.ve} \\  \\ Manuel Saor\'in \footnotemark[1]
\\ Departamento de Matem\'aticas\\
Universidad de Murcia, Aptdo. 4021\\
30100 Espinardo, Murcia\\
SPAIN\\ {\it msaorinc@um.es} }

\begin{document}
\date{}
\maketitle

\footnote{Parra is supported by a grant from the Universidad de los Andes (Venezuela) and Saor\'in is
supported by research projects from the Spanish Ministry of
Education (MTM2010-20940-C02-02) and from the Fundaci\'on 'S\'eneca'
of Murcia (04555/GERM/06), with a part of FEDER funds. The authors
thank these institutions for their help.}

\begin{abstract}
{\bf We study the behavior of direct limits in the heart of a
t-structure. We  prove that, for any compactly generated t-structure
in a triangulated category with  coproducts, countable direct
limits are exact in its heart. Then, for a given Grothendieck
category $\mathcal{G}$ and a torsion pair
$\mathbf{t}=(\mathcal{T},\mathcal{F})$ in $\mathcal{G}$, we show
that  the heart $\mathcal{H}_\mathbf{t}$ of the associated t-structure in the derived
category $\mathcal{D}(\mathcal{G})$ is  AB5 if, and only if, it is a Grothendieck category. If this is the case, then $\mathcal{F}$ is
closed under taking direct limits. The reverse implication is true for a wide
class of torsion pairs which include the hereditary ones, those for
which $\mathcal{T}$ is a cogenerating class and those for which
$\mathcal{F}$ is a generating class. The results allow to extend
results by Buan-Krause and Colpi-Gregorio to the general context of
Grothendieck categories and to improve some results of (co)tilting
theory of modules.}
\end{abstract}

{\bf Mathematics Subjects Classification:} 18E30, 18E15, 18E40,
16E05, 16E30.

\section{Introduction}
Beilinson, Bernstein and Deligne \cite{BBD} introduced the notion of
a t-structure in a triangulated category in their study of perverse
sheaves on an algebraic or analytic variety. If $\mathcal{D}$ is
such a triangulated category, a t-structure is a pair of full
subcategories satisfying suitable axioms (see the precise definition
in next section) which guarantee that their intersection is an
abelian category $\mathcal{H}$, called the heart of the t-structure.
This category  comes with a cohomological functor
$\mathcal{D}\longrightarrow\mathcal{H}$. Roughly speaking, a
t-structure allows to develop an intrinsic (co)homology theory,
where the homology 'spaces' are again objects of $\mathcal{D}$
itself.

Since their introduction t-structures have been used in many
branches of Mathematics, with special impact in Algebraic Geometry
and Representation Theory of Algebras. One line of research in the
topic has been the explicit construction, for concrete triangulated
categories, of wide classes of t-structures. This approach has led
to classification results in many cases  (see, e.g., \cite{Br},
\cite{GKR}, \cite{AJS}, \cite{SS}, \cite{KN}...). A second line of
research consists in starting with a well-behaved class of
t-structures and trying to find necessary and sufficient conditions on
a t-structure in the class so that the heart is a 'nice' abelian
category. For instance, that the heart is a Grothendieck or even a
module category. All the work in this direction which we know of has
concentrated on a particular class of t-structures. Namely, in the
context of bounded derived categories, Happel, Reiten and
Smal$\emptyset$ \cite{HRS} associated to each torsion pair  in an
abelian category $\mathcal{A}$, a t-structure in the bounded derived
category $\mathcal{D}^b(\mathcal{A})$. This t-structure is actually
the restriction of a t-structure
 in $\mathcal{D}(\mathcal{A})$
and several authors (see \cite{CGM}, \cite{CMT}, \cite{MT},
\cite{CG}) have dealt with the problem of deciding when its heart
$\mathcal{H}_\mathbf{t}$ is a Grothendieck or module category, in
case $\mathcal{A}$ is the module category $R-\text{Mod}$, for some
(always associative unital) ring $R$. Concretely, in \cite{CGM} the
authors proved that if $\mathbf{t}$ is faithful in $R-\text{Mod}$,
with $\mathcal{F}$
 closed under taking direct limits,  and $\mathcal{H}_\mathbf{t}$ is
a Grothendieck category, then $\mathbf{t}$ is a cotilting torsion
pair. Later, in \cite{CG} (see also \cite{M}), it was proved that
the converse is also true. The study of the case when
$\mathcal{H}_\mathbf{t}$ is a module category was also initiated in
\cite{CGM} and continued in \cite{CMT}, where the authors gave
necessary and sufficient conditions, when $\mathbf{t}$ is faithful,
for $\mathcal{H}_\mathbf{t}$ to be a module category.

It is clear from the work on the second line of research mentioned
above that the main difficulty in understanding when
$\mathcal{H}_\mathbf{t}$ is a Grothendieck category comes from the
AB5 condition. In fact, the understanding of direct limits in
$\mathcal{H}_\mathbf{t}$ or, more generally, in the heart of any
t-structure is far for complete. The present paper is our first
attempt to understand the AB5 condition on the heart of commonly
used t-structures. We first give some general results for any
t-structure in an arbitrary triangulated category and, later, we
concentrate on the case of a torsion pair $\mathbf{t}$ in any
Grothendieck (not necessarily a module) category. In a forthcoming
paper \cite{PS}, we will study the problem in the case of a
compactly generated t-structure in the derived category
$\mathcal{D}(R)$ of a commutative Noetherian ring, using the
classification results of \cite{AJS}.

We summarize in the following list the main results of the paper,
all of which, except the first,  are given for a torsion pair
$\mathbf{t}=(\mathcal{T},\mathcal{F})$ in a Grothendieck category
$\mathcal{G}$ and the heart $\mathcal{H}_\mathbf{t}$ of the
associated t-structure in $\mathcal{D}(\mathcal{G})$. The reader is
referred to next section for the pertinent definitions.

\begin{enumerate}
\item (Theorem \ref{teor.countable direct limits and compactly
t-structure}) If $\mathcal{D}$ is a triangulated category with coproducts, then, for any compactly generated t-structure
in $\mathcal{D}$, countable direct limits are exact in the heart.
\item (Part of theorem \ref{teor.AB5 heart of a torsion pair})
 The heart $\mathcal{H}_\mathbf{t}$ is a Grothendieck  category if, and only if, it is AB5 if, and only if, the canonical morphism $\varinjlim H^{-1}(M_i)\longrightarrow
H^{-1}(\varinjlim_{\mathcal{H}_\mathbf{t}} M_i)$ is a monomorphism,
for each direct system $(M_i)_{i\in I}$ in $\mathcal{H}_\mathbf{t}$. In this case
$\mathcal{F}$ is closed under taking direct limits in $\mathcal{G}$.

\item (Part of theorem \ref{teor.Grothendieck heart2}) Suppose that
$\mathbf{t}$ satisfies one of the following conditions: i)
$\mathbf{t}$ is hereditary; ii) $\mathcal{F}$ is a generating class;
or iii) $\mathcal{T}$ is a cogenerating class. The heart
$\mathcal{H}_\mathbf{t}$ is a Grothendieck category if, and only if,
$\mathcal{F}$ is closed under taking direct limits in $\mathcal{G}$.

\item (Part of proposition \ref{prop.Grothendieck tilting hearts})
 The torsion class $\mathcal{T}$ is  cogenerating  and $\mathcal{H}_\mathbf{t}$
is a Grothendieck category with a projective generator if, and only
if, $\mathbf{t}$ is a tilting torsion pair such that $\mathcal{F}$
is closed under taking direct limits in $\mathcal{G}$.

\item (Part of proposition \ref{prop.characterization of cotilting
pairs}) When $\mathcal{G}$ is locally finitely presented and
$\mathcal{F}$ is a generating class, the class $\mathcal{F}$ is
closed under taking direct limits in $\mathcal{G}$ if, and only if,
$\mathbf{t}$ is a cotilting torsion pair.
\end{enumerate}

Let us point out that, as a consequence of our findings, some
results in the second line of research mentioned above, as well as
classical results on tilting and cotilting theory of module
categories are extended or improved to more general Grothendieck
categories. For example, Buan-Krause classification of torsion pairs
in the category of finitely generated modules over a Noetherian ring
(\cite{BK}) is extended to any locally noetherian Grothendieck
category (corollary \ref{cor.Buan-Krause}). Similarly, result 5 of
the list above is an extension of the result, essentially due to
Bazzoni, that a faithful torsion pair in $R-\text{Mod}$  is
cotilting if, and only if, its torsionfree class is closed under
direct limits.

The organization of the paper goes as follows. In section 2 we give
all the preliminaries and terminology needed in the rest of the
paper. In section 3 we study Grothendieck properties AB3, AB4, AB5
and their duals, for the heart of a t-structure in any ambient
triangulated category. In particular, we prove result 1 in the list
above.  In Section 4 we give results 2 and 3 in the list above and
their proofs. In the final section 5, we see that the results of the
previous section naturally lead to tilting and cotilting theory in a
general Grothendieck category. Results 4 and 5 in the list above are
proved in this final section, as well as their already mentioned
consequences.

\section{Preliminaries and terminology}

Recall that a category $I$ is  \emph{(skeletally) small} when its
objects form a set (up to isomorphism).  If $\mathcal{C}$ and $I$
are an arbitrary and a small category, respectively, then a functor
$I\longrightarrow\mathcal{C}$ will be called an \emph{$I$-diagram}
on $\mathcal{C}$, or simply a diagram on $\mathcal{C}$ when $I$ is
understood. The category of $I$-diagrams on $\mathcal{C}$ will be
denoted  by $[I,\mathcal{C}]$. We will frequently write an
$I$-diagram $X$ as $(X_i)_{i\in I}$, where $X_i:=X(i)$ for each
$i\in\text{Ob}(I)$, whenever the images by $X$ of the arrows in $I$
are understood.
 When each $I$-diagram has a limit (resp. colimit),
we say that \emph{$\mathcal{C}$ has $I$-limits (resp.
$I$-colimits)}. In such case,
$\text{lim}:[I,\mathcal{C}]\longrightarrow\mathcal{C}$ (resp.
$\text{colim}:[I,\mathcal{C}]\longrightarrow\mathcal{C}$) will
denote the \emph{(I-)limit (resp. ($I$-)colimit) functor} and it is
right (resp. left) adjoint to the constant diagram functor $\kappa
:\mathcal{C}\longrightarrow [I,\mathcal{C}]$. If $\mathcal{C}$ and
$\mathcal{D}$ are categories which have $I$-limits (resp.
$I$-colimits), we will say that a functor
$F:\mathcal{C}\longrightarrow\mathcal{D}$  preserves $I$-limits
(resp. $I$-colimits) when the induced morphism
$F(\text{lim}X_i)\longrightarrow\text{lim}F(X_i)$ (resp.
$\text{colim}F(X_i)\longrightarrow F(\text{colim}X_i)$) is an
isomorphism, for each $I$-diagram $(X_i)_{i\in I}$. The category
$\mathcal{C}$ is called \emph{complete (resp. cocomplete)} when
$I$-limits (resp. $I$-colimits) exist in $\mathcal{C}$, for any
small category $I$.

Recall that a  particular case of limit functor (resp. colimit
functor) is the \emph{($I$-)product functor} $\prod
:[I,\mathcal{C}]\longrightarrow\mathcal{C}$ (resp.
\emph{($I$-)coproduct functor}
$\coprod:[I,\mathcal{C}]\longrightarrow\mathcal{C}$), when $I$ is
just a set, viewed as a small category on which the identities are
its only morphisms. Another particular case comes when  $I$ is a
directed set,  viewed as a small category on which there is a unique
morphism $i\longrightarrow j$ exactly when $i\leq j$. The
corresponding colimit functor is the \emph{$I$-direct limit functor}
$\varinjlim :[I,\mathcal{C}]\longrightarrow\mathcal{C}$. The
$I$-diagrams on $\mathcal{C}$ are usually called \emph{$I$-directed
systems} in $\mathcal{C}$.  Dually, one has  \emph{$I$-inverse
systems} and the \emph{($I$-)inverse limit functor} $\varprojlim
:[I^{op},\mathcal{C}]\longrightarrow\mathcal{C}$.

When
 $\mathcal{A}$ is an additive category and $\mathcal{S}\subset\text{Ob}(\mathcal{A})$ is any class of objects, we shall
denote by $\text{add}_\mathcal{A}(\mathcal{S})$ (resp.
$\text{Add}_\mathcal{A}(\mathcal{S})$), or simply
$\text{add}(\mathcal{S})$ (resp. $\text{Add}(\mathcal{S})$) if no
confusion appears, the class of objects which are direct summands of
finite (resp. arbitrary) coproducts of objects in $\mathcal{S}$.
Also, we will denote by $\text{Prod}_\mathcal{A}(\mathcal{S})$ or
$\text{Prod}(\mathcal{S})$ the class of objects which are direct
summands of arbitrary products of objects in $\mathcal{S}$. When
$\mathcal{S}=\{V\}$, for some object $V$, we will simply write
$\text{add}_\mathcal{A}(V)$ (resp. $\text{Add}_\mathcal{A}(V)$) or
$\text{add}(V)$ (resp. $\text{Add}(V)$) and
$\text{Prod}_\mathcal{A}(V)$ or $\text{Prod}(V)$.  If $\mathcal{S}$
is any set of objects, we will say that it is a \emph{set of
generators} when the functor
$\coprod_{S\in\mathcal{S}}\text{Hom}_\mathcal{A}(S,?):\mathcal{A}\longrightarrow
\text{Ab}$ is a faithful functor. An object $G$ is a
\emph{generator} of $\mathcal{A}$, when $\{G\}$ is a set of
generators. We will employ a stronger version of these concepts, for
a class $\mathcal{R}\subseteq\text{Ob}(\mathcal{A})$. The class
$\mathcal{R}$ will be called a \emph{generating (resp. cogenerating)
class} of $\mathcal{A}$ when, for  each object $X$ of $\mathcal{G}$,
there is  an epimorphism $R\twoheadrightarrow X$ (resp. monomorphism
$X\rightarrowtail R$), for some $R\in\mathcal{R}$.   When, in
addition,
 $\mathcal{A}$ has  coproducts, we shall say that an object $X$ is a
\emph{compact object} when the functor
$\text{Hom}_\mathcal{A}(X,?):\mathcal{A}\longrightarrow\text{Ab}$
preserves coproducts.

Recall  the following 'hierarchy' among abelian categories
introduced by Grothendieck (\cite{Gr}). Let $\mathcal{A}$ be an
abelian category.

\begin{enumerate}
\item[] - $\mathcal{A}$ is \emph{AB3 (resp. AB3*)} when it has
coproducts (resp. products);
\item[] - $\mathcal{A}$ is \emph{AB4 (resp. AB4*)} when it is AB3 (resp. AB3*)
and the coproduct functor $\coprod
:[I,\mathcal{A}]\longrightarrow\mathcal{A}$ (resp. product functor
$\prod :[I,\mathcal{A}]\longrightarrow\mathcal{A}$) is  exact, for
each set $I$;
\item[] - $\mathcal{A}$ is \emph{AB5 (resp. AB5*)} when it is AB3 (resp. AB3*) and the direct limit
functor $\varinjlim :[I,\mathcal{A}]\longrightarrow\mathcal{A}$
(resp. inverse limit functor $\varprojlim
:[I^{op},\mathcal{A}]\longrightarrow\mathcal{A}$) is exact, for each
 directed set $I$.
\end{enumerate}
Note that the AB3 (resp. AB3*) condition is equivalent to the fact
that $\mathcal{A}$ be cocomplete (resp. complete). Note that if
$\mathcal{A}$ is AB3 (resp. AB3*) then,  for each small category
$I$, the corresponding limit (resp. colimit) functor is left (resp.
right) exact, because  it is a right (resp. left) adjoint functor.

An AB5 abelian category $\mathcal{G}$ having a set of generators
(equivalently, a generator),  is called a \emph{Grothendieck
category}. Such a category always has enough injectives,  and even
every object in it has an injective envelope (see \cite{G}).
Moreover, it is always a complete (and cocomplete) category (see
\cite[Corollary X.4.4]{S}). Given an object $V$ of $\mathcal{G}$,
another object $X$ is called \emph{$V$-generated (resp.
$V$-presented)} when there is an epimorphism
$V^{(I)}\twoheadrightarrow X$ (resp. an exact sequence
$V^{(J)}\longrightarrow V^{(I)}\longrightarrow X\rightarrow 0$), for
some sets $I$ and $J$. We will say that $X$ is \emph{$V$-cogenerated
(resp. $V$-copresented)} when  there is a monomorphism
$X\rightarrowtail V^I$ (resp. an exact sequence $0\rightarrow
X\longrightarrow V^I\longrightarrow V^J$), for some sets $I$ and
$J$. As it is customary, we will denote by $\text{Gen}(V)$,
$\text{Pres}(V)$, $\text{Cogen}(V)$ and $\text{Copres}(V)$ the
classes of $V$-generated, $V$-presented, $V$-cogenerated and
$V$-copresented objects, respectively. An object $X$ of
$\mathcal{G}$ is called \emph{finitely presented} when
$\text{Hom}_\mathcal{G}(X,?):\mathcal{G}\longrightarrow\text{Ab}$
preserves direct limits. When $\mathcal{G}$ has a set of finitely
presented generators and each object of $\mathcal{G}$ is a direct
limit of finitely presented objects, we say that $\mathcal{G}$ is
\emph{locally finitely presented}.

A \emph{torsion pair} in an abelian category $\mathcal{A}$ is a pair
$\mathbf{t}=(\mathcal{T},\mathcal{F})$ of full subcategories
satisfying the following two conditions:

\begin{enumerate}
\item[] - $\text{Hom}_\mathcal{A}(T,F)=0$, for all $T\in\mathcal{T}$
and $F\in\mathcal{F}$;
\item[] - For each object $X$ of $\mathcal{A}$ there is an exact
sequence $0\rightarrow T_X\longrightarrow X\longrightarrow
F_X\rightarrow 0$, where $T_X\in\mathcal{T}$ and
$F_X\in\mathcal{F}$.
\end{enumerate}
In such case the objects $T_X$ and $F_X$ are uniquely determined, up
to isomorphism, and the assignment $X\rightsquigarrow T_X$ (resp.
$X\rightsquigarrow F_X$) underlies a functor
$t:\mathcal{A}\longrightarrow\mathcal{T}$ (resp.
$(1:t):\mathcal{A}\longrightarrow\mathcal{F}$) which is right (resp.
left) adjoint to the inclusion functor
$\mathcal{T}\hookrightarrow\mathcal{A}$ (resp.
$\mathcal{F}\hookrightarrow\mathcal{A}$). We will frequently write
$X/t(X)$ to denote $(1:t)(X)$. The torsion pair $\mathbf{t}$ is
called \emph{hereditary} when $\mathcal{T}$ is closed under taking
subobjects in $\mathcal{A}$. Slightly modifying \cite[Definitions
2.3 and 2.6]{C}, when $\mathcal{A}$ is AB3 (resp. AB3*), an object $V$
(resp. $Q$) of $\mathcal{A}$ is called \emph{1-tilting} (resp.
\emph{$1$-cotilting}) when
$\text{Gen}(V)=\text{Ker}(\text{Ext}_\mathcal{A}^1(V,?))$ (resp.
$\text{Cogen}(Q)=\text{Ker}(\text{Ext}_\mathcal{A}^1(?,Q))$). In
that case, one has $\text{Gen}(V)=\text{Pres}(V)$ (resp.
$\text{Cogen}(Q)=\text{Copres}(Q)$) and
$(\text{Gen}(V),\text{Ker}(Hom_\mathcal{A}(V,?)))$ (resp.
$(\text{Ker}(Hom_\mathcal{A}(?,Q),\text{Cogen}(Q))$) is a torsion
pair in $\mathcal{A}$ called the \emph{tilting (resp. cotilting)
torsion pair} associated to $V$ (resp. $Q$).

We refer the reader to \cite{N} for the precise definition of
\emph{triangulated category}, but, diverting from the terminology in
that book, for a given triangulated category $\mathcal{D}$, we will
denote by $?[1]:\mathcal{D}\longrightarrow\mathcal{D}$ its
suspension functor. We will then put $?[0]=1_\mathcal{D}$ and $?[k]$
will denote the $k$-th power of $?[1]$, for each integer $k$.
(Distinguished) triangles in $\mathcal{D}$ will be denoted
$X\longrightarrow Y\longrightarrow Z\stackrel{+}{\longrightarrow}$,
or also $X\longrightarrow Y\longrightarrow
Z\stackrel{w}{\longrightarrow}X[1]$ when the connected morphism $w$
need be emphasized. A \emph{triangulated functor} between
triangulated categories is one which preserves triangles. Unlike the
terminology used in the abstract setting of additive categories, in
the context of triangulated categories a weaker version of the term
'set of generators' is commonly used.  Namely, a set
$\mathcal{S}\subset\text{Ob}(\mathcal{D})$ is called a \emph{set of
generators of $\mathcal{D}$} if an object $X$ of $\mathcal{D}$ is
zero whenever $\text{Hom}_\mathcal{D}(S[k],X)=0$, for all
$S\in\mathcal{S}$ and $k\in\mathbb{Z}$. In case $\mathcal{D}$  has
coproducts, we will say that $\mathcal{D}$ is \emph{compactly
generated} when it has a set of compact generators.

Recall that if $\mathcal{D}$ and $\mathcal{A}$ are a triangulated
and an abelian category, respectively, then an additive  functor
$H:\mathcal{D}\longrightarrow\mathcal{A}$ is a \emph{cohomological
functor} when, given any triangle $X\longrightarrow Y\longrightarrow
Z\stackrel{+}{\longrightarrow}$, one gets an induced long exact
sequence in $\mathcal{A}$:

\begin{center}
$\cdots \longrightarrow H^{n-1}(Z)\longrightarrow H^n(X)\longrightarrow
H^n(Y)\longrightarrow H^n(Z)\longrightarrow
H^{n+1}(X)\longrightarrow \cdots$,
\end{center}
where $H^n:=H\circ (?[n])$, for each $n\in\mathbb{Z}$.

Given a Grothendieck category $\mathcal{G}$, we will denote by
$\mathcal{C}(\mathcal{G})$, $\mathcal{K}(\mathcal{G})$ and
$\mathcal{D}(\mathcal{G})$ the category of chain complexes of
objects of $\mathcal{G}$, the homotopy category of $\mathcal{G}$ and
the derived category of $\mathcal{G}$, respectively (see \cite{V},
\cite{Ke2}).

Let $(\mathcal{D},?[1])$ be a triangulated category. A
\emph{t-structure} in $\mathcal{D}$ is a pair
$(\mathcal{U},\mathcal{W})$ of full subcategories, closed under
taking direct summands in $\mathcal{D}$,  which satisfy the
 following  properties:

\begin{enumerate}
\item[i)] $\text{Hom}_\mathcal{D}(U,W[-1])=0$, for all
$U\in\mathcal{U}$ and $W\in\mathcal{W}$;
\item[ii)] $\mathcal{U}[1]\subseteq\mathcal{U}$;
\item[iii)] For each $X\in Ob(\mathcal{D})$, there is a triangle $U\longrightarrow X\longrightarrow
V\stackrel{+}{\longrightarrow}$ in $\mathcal{D}$, where
$U\in\mathcal{U}$ and $V\in\mathcal{W}[-1]$.
\end{enumerate}
It is easy to see that in such case $\mathcal{W}=\mathcal{U}^\perp
[1]$ and $\mathcal{U}={}^\perp (\mathcal{W}[-1])={}^\perp
(\mathcal{U}^\perp )$. For this reason, we will write a t-structure
as $(\mathcal{U},\mathcal{U}^\perp [1])$. We will call $\mathcal{U}$
and $\mathcal{U}^\perp$ the \emph{aisle} and the \emph{co-aisle} of
the t-structure, respectively. The objects $U$ and $V$ in the above
triangle are uniquely determined by $X$, up to isomorphism, and
define functors
$\tau_\mathcal{U}:\mathcal{D}\longrightarrow\mathcal{U}$ and
$\tau^{\mathcal{U}^\perp}:\mathcal{D}\longrightarrow\mathcal{U}^\perp$
which are right and left adjoints to the respective inclusion
functors. We call them the \emph{left and right truncation functors}
with respect to the given t-structure. Note that
$(\mathcal{U}[k],\mathcal{U}^\perp [k])$ is also a t-structure in
$\mathcal{D}$, for each $k\in\mathbb{Z}$. The full subcategory
$\mathcal{H}=\mathcal{U}\cap\mathcal{W}=\mathcal{U}\cap\mathcal{U}^\perp
[1]$ is called the \emph{heart} of the t-structure and it is an
abelian category, where the short exact sequences 'are' the
triangles in $\mathcal{D}$ with its three terms in $\mathcal{H}$.
Moreover, with the obvious abuse of notation,  the assignments
$X\rightsquigarrow (\tau_{\mathcal{U}}\circ\tau^{\mathcal{U}^\perp
[1]})(X)$ and $X\rightarrow (\tau^{\mathcal{U}^\perp
[1]}\circ\tau_\mathcal{U})(X)$ define   naturally isomorphic
functors $\mathcal{D}\longrightarrow\mathcal{H}$ whih are
cohomological (see \cite{BBD}).  When $\mathcal{D}$ has coproducts,
the t-structure $(\mathcal{U},\mathcal{U}^\perp [1])$ is called
\emph{compactly generated} when there is a set
$\mathcal{S}\subseteq\mathcal{U}$, consisting of compact objects,
such that $\mathcal{U}^\perp$ consists of the $Y\in\mathcal{D}$ such
that $\text{Hom}_\mathcal{D}(S[n],Y)=0$, for all $S\in\mathcal{S}$
and integers $n\geq 0$. In such case, we say that $\mathcal{S}$ is a
\emph{set of compact generators of the aisle $\mathcal{U}$}.

\begin{exems} \label{exems.two canonical ones}
The following are typical examples t-structures:

\begin{enumerate}
\item Let $\mathcal{G}$ be a Grothendieck category and, for each $k\in\mathbb{Z}$,  denote by $\mathcal{D}^{\leq
k}(\mathcal{G})$ (resp. $\mathcal{D}^{\geq k}(\mathcal{G})$) the
full subcategory of $\mathcal{D}(\mathcal{G})$ consisting of the
complexes $X$ such that $H^j(X)=0$, for all $j>k$ (resp. $j<k$). The
pair $(\mathcal{D}^{\leq
k}(\mathcal{G}),\mathcal{D}^{\geq k}(\mathcal{G}))$ is a t-structure
in $\mathcal{D}(\mathcal{G})$ whose heart is equivalent to
$\mathcal{G}$. Its left and right truncation functors will be
denoted by $\tau^{\leq
k}:\mathcal{D}(\mathcal{G})\longrightarrow\mathcal{D}^{\leq
k}(\mathcal{G})$ and
$\tau^{>k}:\mathcal{D}(\mathcal{G})\longrightarrow\mathcal{D}^{>k}(\mathcal{G}):=\mathcal{D}^{\geq
k}(\mathcal{G})[-1]$. For $k=0$, the given t-structure is known as
the \emph{canonical t-structure} in $\mathcal{D}(\mathcal{G})$.
\item (Happel-Reiten-Smal$\emptyset$) Let $\mathcal{G}$ be any
Grothendieck category and $\mathbf{t}=(\mathcal{T},\mathcal{F})$ be
a torsion pair in $\mathcal{G}$. One gets a t-structure
$(\mathcal{U}_\mathbf{t},\mathcal{U}_\mathbf{t}^{\perp}[1])=(\mathcal{U}_\mathbf{t},\mathcal{W}_\mathbf{t})$
in $\mathcal{D}(\mathcal{G})$, where:

\begin{center}
$\mathcal{U}_\mathbf{t}=\{X\in\mathcal{D}^{\leq 0}(\mathcal{G}):$
$H^0(X)\in\mathcal{T}\}$

$\mathcal{W}_\mathbf{t}=\{Y\in\mathcal{D}^{\geq -1}(\mathcal{G}):$
$H^{-1}(Y)\in\mathcal{F}\}$.
\end{center}
In this case, the heart $\mathcal{H}_\mathbf{t}$ consists of the
complexes $M$ such that $H^{-1}(M)\in\mathcal{F}$,
$H^0(M)\in\mathcal{T}$ and $H^k(M)\neq 0$, for all $k\neq -1,0$.
Each such complex is  isomorphic in $\mathcal{D}(\mathcal{G})$ to a
complex  $\cdots \longrightarrow 0\longrightarrow X\stackrel{d}{\longrightarrow}
Y\longrightarrow 0 \longrightarrow \cdots $, concentrated in degrees $-1$ and $0$, such
that $\text{Ker}(d)\in\mathcal{F}$ and
$\text{Coker}(d)\in\mathcal{T}$.

\item Let $\mathcal{D}$ be a triangulated category which has
coproducts. An object $X$ of $\mathcal{D}$ will be called a
\emph{tilting complex} when $\{X\}$ is a set of compact generators
of $\mathcal{D}$ such that $\text{Hom}_\mathcal{D}(X,X[n])=0$, for
all $n\in\mathbb{Z}\setminus\{0\}$. When $\mathcal{D}$ is the base
of a derivator (see \cite[Appendix 1]{KN}), in particular when
$\mathcal{D}=\mathcal{D}(\mathcal{G})$ for a Grothendieck category
$\mathcal{G}$, the smallest full subcategory $\mathcal{U}_X$ of
$\mathcal{D}$ which contains $X$ and is closed under coproducts,
extensions and application of the functor $?[1]$ is the  aisle of a
t-structure whose co-aisle is $\mathcal{U}_X^\perp
=\{Y\in\text{Ob}(\mathcal{D}):$
$\text{Hom}_\mathcal{D}(X[n],Y)=0\text{, for all }n\geq 0\}$ (see
also \cite{AJSo}). The corresponding heart $\mathcal{H}_X$ is
equivalent to the module category over the ring
$R:=\text{End}_{\mathcal{D}}(X)$ and the equivalence of categories
$\text{Hom}_\mathcal{D}(X,?):\mathcal{H}_X\stackrel{\cong}{\longrightarrow}R-\text{Mod}$
extends to a triangulated equivalence
$\mathcal{D}\stackrel{\cong}{\longrightarrow}{D}(R)$.
\end{enumerate}

\end{exems}

\section{Colimits in the heart of a t-structure}

In the sequel $(\mathcal{D},?[1])$ is a triangulated category. All
throughout this section, we will fix a t-structure
$(\mathcal{U},\mathcal{U}^\perp [1])$ in $\mathcal{D}$ and
$\mathcal{H}=\mathcal{U}\cap\mathcal{U}^\perp [1]$ will be its
heart. We will denote  by
$\tilde{H}:\mathcal{D}\longrightarrow\mathcal{H}$ either of the
naturally isomorphic functors
$\tau_\mathcal{U}\circ\tau^{\mathcal{U}^\perp[1]}$ or
$\tau^{\mathcal{U}^\perp [1]} \circ \tau_\mathcal{U} $, in order to avoid confusion, in case
$\mathcal{D}=\mathcal{D}(\mathcal{G})$, with the classical
cohomological functor
$H=H^0:\mathcal{D}(\mathcal{G})\longrightarrow\mathcal{G}$.

\begin{lemma} \label{lem.adjoints to the inclusions from heart}
The following assertions hold:

\begin{enumerate}
\item If $X$ is an object of $\mathcal{U}$, then $\tilde{H}(X)\cong \tau^{\mathcal{U}^\perp
}(X[-1])[1]$ and the assignment $X\rightsquigarrow\tilde{H}(X)$
defines an additive functor
$L:\mathcal{U}\longrightarrow\mathcal{H}$ which is left adjoint to
the inclusion $j:\mathcal{H}\hookrightarrow\mathcal{U}$.

\item If $Y$ is an object of $\mathcal{U}^\perp[1]$, then $\tilde{H}(Y)\cong \tau_{\mathcal{U}}(Y)$ and the assignment
$Y\rightsquigarrow\tilde{H}(Y)$ defines an additive functor
$R:\mathcal{U}^{\perp}[1]\longrightarrow\mathcal{H}$ which is right
adjoint to the inclusion
$j:\mathcal{H}\hookrightarrow\mathcal{U}^\perp [1]$.
\end{enumerate}
\end{lemma}
\begin{proof}
Suppose that $X,U\in\mathcal{U}$ and that $V\in\mathcal{U}^\perp$. A
sequence of morphisms

\begin{center}
$U\stackrel{}{\longrightarrow}X[-1]\stackrel{g}{\longrightarrow}V\stackrel{h}{\longrightarrow}U[1]$
\end{center}
is a distinguished triangle if, and only if, the sequence

\begin{center}
$U[1]\stackrel{}{\longrightarrow}X\stackrel{g[1]}{\longrightarrow}V[1]\stackrel{h[1]}{\longrightarrow}U[2]$
\end{center}
is so. It follows from this that $\tau^{\mathcal{U}^\perp
[1]}(X)=\tau^{\mathcal{U}^\perp}(X[-1])[1]$, and then the
isomorphism $\tilde{H}(X)\cong \tau^{\mathcal{U}^\perp}(X[-1])[1]$
follows from the definition of $\tilde{H}$. On the other hand, if
$Y\in\mathcal{U}^\perp [1]$ then, by the definition of $\tilde{H}$,
we get $\tilde{H}(Y)\cong\tau_\mathcal{U}(Y)$.

The part of assertion 1 relative to the adjunction is dual to that
of assertion 2 since $(\mathcal{U},\mathcal{U}^\perp [1])$ is a
t-structure in $\mathcal{D}$ exactly when
$(\mathcal{U}^\perp[1],\mathcal{U})$ is a t-structure in
$\mathcal{D}^{op}$. We then prove the adjunction of assertion 2,
which follows directly from the following chain of isomorphisms,
using the fact that
$\tau_\mathcal{U}:\mathcal{D}\longrightarrow\mathcal{U}$ is right
adjoint to the inclusion
$j_\mathcal{U}:\mathcal{U}\hookrightarrow\mathcal{D}$:

\begin{center}
$\text{Hom}_{\mathcal{U}^\perp
[1]}(j(Z),Y)=\text{Hom}_\mathcal{D}(j(Z),Y)\cong\text{Hom}_\mathcal{U}(j(Z),\tau_\mathcal{U}(Y))\cong\text{Hom}_\mathcal{H}(Z,\tilde{H}(Y))=
\text{Hom}_\mathcal{H}(Z,R(Y))$.
\end{center}
\end{proof}

\begin{prop} \label{prop.AB3 condition}
Let  $\mathcal{D}$ be a triangulated category which has coproducts
(resp. products) and let $(\mathcal{U},\mathcal{U}^\perp[1])$ be a
t-structure in $\mathcal{D}$. The heart $\mathcal{H}$ is an AB3
(resp. AB3*) abelian category.
\end{prop}
\begin{proof}
 It is a known fact and very easy to prove that if
 $L:\mathcal{C}\longrightarrow\mathcal{C}'$ is a left adjoint
 functor  and $\mathcal{C}$ has  coproducts, then, for each
 family $(C_i)_{i\in I}$ of objects of $\mathcal{C}$, the family $(L(C_i))_{i\in
 I}$ has a coproduct in $\mathcal{C}'$ and one has an isomorphism $\coprod_{i\in I}L(C_i)\cong L(\coprod_{i\in
 I}C_i)$.

 Let now $(Z_i)_{i\in I}$ be a family of objects of $\mathcal{H}$.
 Since the counit $L\circ j\rightarrow 1_\mathcal{H}$ is an isomorphism, we have that $(L\circ j)(Z_i)\cong
 Z_i$ and the fact that the family has a coproduct in $\mathcal{H}$
 follows directly from the previous paragraph.

 The statement about products is dual to the one for coproducts.
\end{proof}

 The following is an interesting  type of t-structures.

\begin{defi} \label{def.(co)smashing t-structure}
Let us assume that $\mathcal{D}$ has coproducts (resp. products). The t-structure
$(\mathcal{U},\mathcal{U}^\perp [1])$ is called \emph{smashing}
(resp. \emph{co-smashing}) when $\mathcal{U}^\perp$ (resp.
$\mathcal{U}$) is closed under taking  coproducts (resp. products)
in $\mathcal{D}$. Bearing in mind that coproducts (resp. products) of triangles are triangles (see \cite[Proposition 1.2.1]{N}), this is equivalent to saying that  the left (resp. right)
truncation functor
$\tau_\mathcal{U}:\mathcal{D}\longrightarrow\mathcal{U}$ (resp.
$\tau^{\mathcal{U}^\perp}:\mathcal{D}\longrightarrow\mathcal{U}^\perp$)
preserves  coproducts (resp. products).

\end{defi}

\begin{prop} \label{prop.AB4 condition}
Let $\mathcal{D}$ be a triangulated category that has coproducts
(resp. products). If
$(\mathcal{U},\mathcal{U}^\perp [1])$ is a t-structure whose heart
$\mathcal{H}$ is closed under taking  coproducts (resp.
products) in $\mathcal{D}$, then $\mathcal{H}$ is an AB4 (resp.
AB4*) abelian category. In particular, that happens when
$(\mathcal{U},\mathcal{U}^\perp [1])$ is a smashing (resp.
co-smashing) t-structure.
\end{prop}
\begin{proof}
Note that $\mathcal{U}$ (resp. $\mathcal{U}^\perp $) is closed under
taking  coproducts (resp. products) in $\mathcal{D}$. Then the final
assertion follows automatically from the first part of the
proposition and from the definition of smashing (resp. co-smashing)
t-structure.

We just do the AB4 case since the AB4* one is dual.  Let
$(0\rightarrow X_i\longrightarrow Y_i\longrightarrow Z_i\rightarrow
0)_{i\in I}$ be a family of short exact sequences in $\mathcal{H}$.
According to \cite{BBD}, they come from triangles in $\mathcal{D}$.
By \cite[Proposition 1.2.2]{N}, we get a triangle in $\mathcal{D}$

\begin{center}
$\coprod_{i\in I}X_i\longrightarrow\coprod_{i\in
I}Y_i\longrightarrow\coprod_{i\in
I}Z_i\stackrel{+}{\longrightarrow}$,
\end{center}
where the three terms are in $\mathcal{H}$ since $\mathcal{H}$  is
closed under taking  coproducts in $\mathcal{D}$. We then get the
desired short exact sequence $0\rightarrow \coprod_{i\in
I}X_i\longrightarrow\coprod_{i\in I}Y_i\longrightarrow\coprod_{i\in
I}Z_i\rightarrow 0$ in $\mathcal{H}$.
\end{proof}

\begin{defi} \label{def.cohomological datum}
Let $\mathcal{X}$ be any full subcategory of $\mathcal{D}$.  A
\emph{cohomological datum in $\mathcal{D}$ with respect to
$\mathcal{X}$} is a pair $(H,r)$ consisting of a cohomological
functor $H:\mathcal{D}\longrightarrow\mathcal{A}$, where
$\mathcal{A}$ is an abelian category, and $r$ is an element of
$\mathbb{Z}\cup\{+\propto\}$ such that the family of functor
$(H^k_{|\mathcal{X}}:\mathcal{X}\hookrightarrow\mathcal{D}\stackrel{H^k}{\longrightarrow}\mathcal{A})_{k<r}$
is conservative. That is, if $X\in \mathcal{X}$ and $H^k(X)=0$, for
all $k<r$, then $X=0$.
\end{defi}

The following is an useful result inspired by \cite[Theorem
3.7]{CGM}:

\begin{prop} \label{prop.exactness of I-colimits}
Suppose that $\mathcal{D}$ has coproducts   and let
$(H:\mathcal{D}\longrightarrow\mathcal{A},r)$ be a cohomological
datum in $\mathcal{D}$ with respect to the heart
$\mathcal{H}=\mathcal{U}\cap\mathcal{U}^\perp [1]$. Suppose that $I$
is a small category  such that $I$-colimits exist and are exact in
$\mathcal{A}$. If, for each  diagram $X:I\longrightarrow\mathcal{H}$
and each integer $k<r$, the canonical morphism
$\text{colim}H^k(X_i)\longrightarrow H^k(colim_{\mathcal{H}}(X_i))$
is an isomorphism, then $I$-colimits are exact in $\mathcal{H}$.
\end{prop}
\begin{proof}
By proposition \ref{prop.AB3 condition}, we know that $\mathcal{H}$
is AB3 or, equivalently, cocomplete. Let us consider an $I$-diagram
$0\rightarrow X_i\stackrel{f_i}{\longrightarrow}
Y_i\stackrel{g_i}{\longrightarrow} Z_i\rightarrow 0$ of short exact
sequences in $\mathcal{H}$. Formally speaking, this is just a
functor $I\longrightarrow\mathcal{H}_{sec}$, where
$\mathcal{H}_{sec}$ denotes the  category of short exact sequences
in $\mathcal{H}$. By right exactness of colimits, we
 then get
an exact sequence in $\mathcal{H}$:

\begin{center}
$\text{colim}_\mathcal{H}(X_i)\stackrel{f}{\longrightarrow}\text{colim}_\mathcal{H}(Y_i)
\stackrel{g}{\longrightarrow}\text{colim}_\mathcal{H}(Z_i)\rightarrow
0$.
\end{center}
We put $L:=\text{Im}(f)$ and then consider the two induced short
exact sequences in  $\mathcal{H}$:

\begin{center}
$0\rightarrow L\longrightarrow
\text{colim}_\mathcal{H}(Y_i)\stackrel{g}{\longrightarrow}\text{colim}_\mathcal{H}(Z_i)\rightarrow
0$

$0\rightarrow W\longrightarrow
\text{colim}_\mathcal{H}(X_i)\stackrel{p}{\longrightarrow}L\rightarrow
0$.
\end{center}

We view all given short exact sequences in $\mathcal{H}$ as
triangles in $\mathcal{D}$ and use the cohomological condition of
$H$ and the fact that $I$-colimits are exact in $\mathcal{A}$, and
get the following commutative diagram in $\mathcal{A}$ with exact
rows.

$$\xymatrix{\text{colim}H^{k-1}(Y_{i}) \ar[r] \ar[d]^{\wr} & \text{colim}H^{k-1}(Z_{i}) \ar[r] \ar[d]^{\wr} & \text{colim}H^{k}(X_{i}) \ar[r] \ar[d] & \text{colim}H^{k}(Y_{i}) \ar[r] \ar[d]^{\wr} & \text{colim}H^{k}(Z_{i}) \ar[d]^{\wr} \\ H^{k-1}(\text{colim}_{\mathcal{H}}(Y_{i})) \ar[r] & H^{k-1}(\text{colim}_{\mathcal{H}}(Z_{i})) \ar[r] & H^{k}(L) \ar[r] & H^{k}(\text{colim}_{\mathcal{H}}(Y_{i})) \ar[r] & H^{k}(\text{colim}_{\mathcal{H}}(Z_{i}))}$$

where the vertical arrow $\text{colim}H^k(X_i)\longrightarrow
H^k(L)$ is the composition $\text{colim}H^k(X_i)\longrightarrow
H^k(\text{colim}_{\mathcal{H}}(X_i))\stackrel{H^k(p)}{\longrightarrow}H^k(L)$,
for each $k\in\mathbb{Z}$. For $k<r$, in principle, all the vertical
arrows except the central one are isomorphisms. Then also the
central one is an isomorphism, which implies that $H^k(p)$ is an
isomorphism since, by hypothesis, the canonical morphism
$\text{colim}H^k(X_i)\longrightarrow
H^k(\text{colim}_{\mathcal{H}}(X_i))$ is an isomorphism. We then get
that $H^k(W)=0$, for all $k<r$, which implies that $W=0$ due to
definition \ref{def.cohomological datum}. Therefore $p$ is an
isomorphism.
\end{proof}

\begin{exem} \label{exem.compact cohomological datum}
If $(\mathcal{U},\mathcal{U}^\perp [1])$ be a compactly generated
t-structure and let  $\mathcal{S}$ be a set of compact generators of
its aisle.  Then
$H:=\coprod_{S\in\mathcal{S}}\text{Hom}_\mathcal{D}(S,?):\mathcal{D}\longrightarrow\text{Ab}$
is a cohomological functor. Moreover,   the pair $(H,1)$ is a
cohomological datum with respect to the heart
$\mathcal{H}=\mathcal{U}\cap\mathcal{U}^\perp [1]$.
\end{exem}

Given a sequence

\begin{center}
$X_0\stackrel{f_1}{\longrightarrow}X_1\stackrel{f_2}{\longrightarrow}\cdots \stackrel{f_n}{\longrightarrow}X_n\stackrel{f_{n+1}}{\longrightarrow} \cdots $
\end{center}
of morphisms in  $\mathcal{D}$, we will call  \emph{Milnor colimit}
of the sequence, denoted $\text{Mcolim}X_n$, what is  called
homotopy colimit in \cite{N}.

\begin{lemma} \label{lem.Milnor Colimits}
Suppose that $\mathcal{D}$ has coproducts and that
$(\mathcal{U},\mathcal{U}^{\perp}[1])$ is a compactly generated
t-structure in $\mathcal{D}$. Then $\mathcal{U}^{\perp}$ is closed
under taking Milnor colimits.
\end{lemma}
\begin{proof}
Let $\mathcal{D}^c$ be the full subcategory of compact objects and
take  $C\in \mathcal{D}^c$ arbitrary. We claim that, for any diagram
in $\mathcal{D}$ of the form:
$$\xymatrix{(\ast) & X_0 \ar[r]^{f_1} & X_1 \ar[r]^{f_2} & X_2 \ar[r]^{f_3}  & \cdots },$$
we have an isomorphism $Hom_{\mathcal{D}}(C,Mcolim(X_n))\cong
\varinjlim{Hom_{\mathcal{D}}(C,X_n)}$. To see that, let us consider
the canonical triangle
$$\xymatrix{\coprod_{n\geq 0} X_n \ar[r]^{1-f \hspace{0.1 cm}} & \coprod_{n\geq 0} X_n \ar[r] & Mcolim(X_{n}) \ar[r]^{\hspace{0.8cm +}} &}$$
if $C\in \mathcal{D}^{c}$ then we have that the following diagram is
commutative:
$$\xymatrix{\coprod_{n \geq 0} Hom_{\mathcal{D}}(C,X_n) \ar[r]^{1-f_{\ast}} \ar[d]^{\wr} & \coprod_{n \geq 0} Hom_{\mathcal{D}}(C,X_n) \ar[d]^{\wr} \\ Hom_{\mathcal{D}}(C,\coprod_{n\geq 0}X_n) \ar[r]^{(1-f)^{\ast}} & Hom_{\mathcal{D}}(C,\coprod_{n \geq 0}X_n) }$$
Note that $1-f_{\ast}$ is a monomorphism in $Ab$. Then we get an
exact sequence in $Ab$ of the form:
$$\xymatrix{\cdots \ar[r]^{0 \hspace{1.5 cm}} & Hom_{\mathcal{D}}(C, \coprod_{n \geq 0} X_n) \ar[r] & Hom_{\mathcal{D}}(C, \coprod_{n \geq 0} X_n) \ar[r] & Hom_{\mathcal{D}}(C,Mcolim(X_n)) \ar@(d,ur)[dl]_{0} \\ &&  Hom_{\mathcal{D}}(C[-1], \coprod_{n \geq 0} X_n) \ar[r] &  Hom_{\mathcal{D}}(C[-1], \coprod_{n \geq 0} X_n) \ar[r] & \cdots}$$
This proves the claim since
$\text{Coker}(1-f_\ast)=\varinjlim\text{Hom}_\mathcal{D}(C,X_n)$.
Now if all the $X_n$ are in $\mathcal{U}^{\perp}$, then for each
$C\in \mathcal{U}\cap \mathcal{D}^{c}$ and each $k\geq 0$, we have
$Hom_{\mathcal{D}}(C[k],Mcolim(X_n)) \cong
\varinjlim{Hom_{\mathcal{D}}(C[k],X_n)}=0$. This shows that
$\text{Mcolim}(X_n)\in\mathcal{U}^\perp$ since the t-structure is
compactly generated.
\end{proof}

\begin{teor}\label{teor.countable direct limits and compactly t-structure}
Suppose that $\mathcal{D}$ has coproducts and let
$(\mathcal{U},\mathcal{U}^{\perp}[1])$ be a compactly generated
t-structure in $\mathcal{D}$. Countable direct limits are exact in
$\mathcal{H}=\mathcal{U}\cap \mathcal{U}^{\perp}[1]$.
\end{teor}
\begin{proof}
Let $I$ be a countable directed set. Then there is an ascending
chain of finite directed subposets $I_0\subset I_1\subset I_2 \dots
$ such that $I=\cup_{n\in \mathbb{N}} I_n$ (see \cite[Lemma
1.6]{AR}). If we put $i_{n}=max\{I_n\}$, for all $n\in \mathbb{N}$,
we clearly have $i_{n} \leq i_{n+1}$, for all $n$ and
$J=\{i_0,i_1,\ldots,i_n,\ldots\}$ is a cofinal subset of $I$ which
isomorphic to $\mathbb{N}$ as an ordered set. Then we know that, for
any category $\mathcal{C}$ with direct limits, the diagram
$$\xymatrix{ [J,\mathcal{C}] \ar[r]^{colim_J \hspace{0.2 cm}} & colim_{J}\mathcal{C}  \\  [I,\mathcal{C}] \ar[ur]_{colim_I} \ar[u]^{restriction}& }$$
is commutative.

We can then assume that $I=\mathbb{N}$. But the previous lemma tells
us that if $X:\mathbb{N}\longrightarrow\mathcal{H}$
($n\rightsquigarrow X_n$) is any diagram, then the triangle
$$\xymatrix{\coprod_{n\geq 0} X_n \ar[r]^{1-f \hspace{0.2 cm}} & \coprod_{n \geq 0} X_n \ar[r] & Mcolim(X_n) \ar[r]^{\hspace{0.8cm}+} &}$$
lives in $\mathcal{H}$ and, hence, it is an exact sequence in this
abelian category. Therefore $Mcolim(X_n)\cong
\varinjlim_{\mathcal{H}}{X_n}$. If now $\mathcal{S}$ is any set of
compact generators of the aisle $\mathcal{U}$ and we take the
cohomological functor
$H:=\coprod_{S\in\mathcal{S}}\text{Hom}_\mathcal{D}(S,?):\mathcal{D}\longrightarrow\text{Ab}$,
then, by the proof of the previous lemma,  the induced map
$\varinjlim H^k(X_n)\longrightarrow H^k (\varinjlim_\mathcal{H}X_n)$
is an isomorphism, for each $k\in\mathbb{Z}$. The result now follows
from proposition \ref{prop.exactness of I-colimits} and example
\ref{exem.compact cohomological datum}.
\end{proof}

In view of last result, the following is  a natural question:

\begin{ques}
Let $\mathcal{D}$ be a triangulated category and
$(\mathcal{U},\mathcal{U}^\perp [1])$ be a smashing t-structure in
$\mathcal{D}$. Is its heart $\mathcal{H}$ an AB5 abelian category?.
Is it so when  the t-structure is compactly generated?
\end{ques}

In next section we tackle this question for the (smashing)
t-structure in $\mathcal{D}(\mathcal{G})$ defined by a torsion pair
in  the Grothendieck category $\mathcal{G}$. In a forthcoming paper
\cite{PS}, we will settle it for essentially all the compactly
generated t-structures in $\mathcal{D}(R)$, when $R$ is a
commutative Noetherian ring.

\section{When is the heart of a torsion pair a Grothendieck category?}

All throughout this section $\mathcal{G}$ is a Grothendieck category
and $\mathbf{t}=(\mathcal{T},\mathcal{F})$ is a torsion pair in
$\mathcal{G}$. We will follow the terminology and notation
introduced in example \ref{exems.two canonical ones}(2). Note that
$\mathcal{T}$ is closed under taking direct limits in $\mathcal{G}$,
while $\mathcal{F}$ need not be so. Note also that
$(\mathcal{U}_\mathbf{t},\mathcal{W}_\mathbf{t})=(\mathcal{U}_\mathbf{t},\mathcal{U}_\mathbf{t}^\perp
[1])$ is a smashing t-structure in $\mathcal{D}(\mathcal{G})$, so
that, by proposition \ref{prop.AB4 condition}, the heart
$\mathcal{H}_\mathbf{t}$ is always an AB4 abelian category.

\begin{lemma} \label{lem.exactness of the Hi}
Let $\mathbf{t}=(\mathcal{T},\mathcal{F})$ be a torsion pair in the
Grothendieck category $\mathcal{G}$, let
$(\mathcal{U_\mathbf{t}},\mathcal{U}_\mathbf{t}^\perp [1])$ be its
associated t-structure (see example \ref{exems.two canonical
ones}(2)) and let $\mathcal{H}_\mathbf{t}$ be its heart. The functor
$H^0:\mathcal{H}_\mathbf{t}\longrightarrow\mathcal{G}$ is right
exact while the functor
$H^{-1}:\mathcal{H}_\mathbf{t}\longrightarrow\mathcal{G}$ is left
exact. Both of them preserves coproducts.

\end{lemma}
\begin{proof}
 The functor $H^k$ vanishes on $\mathcal{H}_\mathbf{t}$, for each
$k\neq -1,0$. By applying now the long exact sequence of homologies
to any short exact sequence in $\mathcal{H}_\mathbf{t}$ the right
(resp. left) exactness of $H^0$ (resp. $H^{-1}$) follows
immediately. Since coproducts in $\mathcal{H}_\mathbf{t}$ are
calculated as in $\mathcal{D}(\mathcal{G})$, the fact that  $H^0$
and $H^{-1}$ preserve coproducts is clear.
\end{proof}

\begin{defi} \label{def.colimit-defining morphism}
Let $I$ be a directed set and $\mathcal{C}$ be any cocomplete
category. Given an $I$-directed system $[(X_i)_{i\in
I},(u_{ij})_{i\leq j}]$, we put $X_{ij}=X_i$, for all $i\leq j$. The
\emph{colimit-defining morphism} associated to the direct system is
the unique morphism $f:\coprod_{i\leq
j}X_{ij}\longrightarrow\coprod_{i\in I}X_i$ such that if
$\lambda_{kl}:X_{kl}\longrightarrow\coprod_{i\leq j}X_{ij}$ and
$\lambda_j:X_j\longrightarrow\coprod_{i\in I}X_i$ are the canonical
morphisms into the coproducts, then
$f\circ\lambda_{ij}=\lambda_i-\lambda_j\circ u_{ij}$ for all $i\leq
j$.
\end{defi}
By classical category theory (see, e.g. \cite[Proposition
II.6.2]{P}), in the situation of last definition, we have that
$L:=\varinjlim X_i\cong\text{Coker}(f)$.

The following is the crucial result for our purposes. In its
statements and all throughout the rest of the paper, unadorned
direct limits are considered in $\mathcal{G}$, while we will denote
by  $\varinjlim_{\mathcal{H}_\mathbf{t}}$ the direct limit in
$\mathcal{H}_\mathbf{t}$.

\begin{prop} \label{prop.direct limit under H0 and H1}
Let $\mathbf{t}=(\mathcal{T},\mathcal{F})$ be a torsion pair in the
Grothendieck category $\mathcal{G}$ and let $\mathcal{H}_\mathbf{t}$
be the heart of the associated t-structure in
$\mathcal{D}(\mathcal{G})$. The following assertions hold:

\begin{enumerate}
\item If $(M_i)_{i\in I}$ is a direct system
in $\mathcal{H}_\mathbf{t}$, then the induced morphism $\varinjlim
H^k(M_i)\longrightarrow H^k(\varinjlim_\mathcal{H_\mathbf{t}}M_i)$
is an epimorphism, for $k=-1$, and an isomorphism, for $k\neq -1$.

\item If $(F_i)_{i\in I}$ is a direct system in $\mathcal{F}$, there
is an isomorphism $(\frac{\varinjlim F_i}{t(\varinjlim
F_i)})[1]\cong\varinjlim_\mathcal{H_\mathbf{t}}(F_i[1])$ in
$\mathcal{H}_\mathbf{t}$.

\item If $(T_i)_{i\in I}$ is a direct system in $\mathcal{T}$, then
the following conditions hold true:

\begin{enumerate}
\item The induced morphism $\varinjlim_{\mathcal{H}_\mathbf{t}}(T_i[0])\longrightarrow (\varinjlim
T_i)[0]$ is an isomorphism in $\mathcal{H}_\mathbf{t}$;
\item The kernel of the colimit-defining morphism $f:\coprod_{i\leq j}T_{ij}\longrightarrow\coprod_{i\in
I}T_i$ is in $\mathcal{T}$.
\end{enumerate}

\end{enumerate}
\end{prop}
\begin{proof}
1) It  essentially follows from \cite[Corollary 3.6]{CGM}, but, for
the sake of completeness, we give a short proof. By lemma
\ref{lem.exactness of the Hi}, we know that
$H^0:\mathcal{H_\mathbf{t}}\longrightarrow\mathcal{G}$ preserves
colimits, in particular direct limits. Let $f:\coprod_{i\leq
j}M_{ij}\longrightarrow\coprod_{i\in I}M_i$ be the associated
colimit-defining morphism and denote by $W$ the image of $f$ in
$\mathcal{H}_\mathbf{t}$. Applying the exact sequence of homology to
the exact sequence $0\rightarrow W\longrightarrow\coprod_{i\in
I}M_i\stackrel{p}{\twoheadrightarrow}L\rightarrow 0$, where
$L=\varinjlim_\mathcal{H_\mathbf{t}}M_i$,  and using \cite[Lemma
3.5]{CGM}, we readily see that we have a short exact sequence
$0\rightarrow H^{-1}(W)\longrightarrow \coprod_{i\in
I}H^{-1}(M_i)\stackrel{H^{-1}(p)}{\longrightarrow}H^{-1}(L)\rightarrow
0$. But the last arrow in this sequence is the composition

\begin{center}
$\coprod_{i\in I}H^{-1}(M_i)\twoheadrightarrow\varinjlim
H^{-1}(M_i)\stackrel{can}{\twoheadrightarrow}H^{-1}(\varinjlim_\mathcal{H_\mathbf{t}}M_i)$,
\end{center}
whose second arrow is then an epimorphism.

\vspace*{0.3cm}

In order to prove the remaining assertions, we first consider any
direct system $(M_i)_{i\in I}$ in $\mathcal{H}_\mathbf{t}$ and the
associated triangle given by the colimit-defining morphism:

\begin{center}
$\coprod_{i\leq j}M_{ij}\stackrel{f}{\longrightarrow}\coprod_{i\in
I}M_i\stackrel{q}{\longrightarrow} Z\stackrel{+}{\longrightarrow}$
\end{center}
in $\mathcal{D}(\mathcal{G})$. We claim that if $v:Z\longrightarrow
L:=\varinjlim_{\mathcal{H}_\mathbf{t}}(M_i)$ is a morphism fitting
in a triangle $N[1]\longrightarrow
Z\stackrel{v}{\longrightarrow}L\stackrel{+}{\longrightarrow}$, with
$N\in\mathcal{H}_\mathbf{t}$, then
 $H^{-1}(v)$ induces an isomorphism
$H^{-1}(Z)/t(H^{-1}(Z))\stackrel{\cong}{\longrightarrow}H^{-1}(L)$.

  By \cite{BBD}, we have a diagram:

$$\xymatrix{&&N[1] \ar[d] &\\ \coprod_{i\leq j} M_{ij} \ar[r]^{f} & \coprod_{i\in I} M_{i} \ar[r] & Z \ar[d] \ar[r]^{+} & \\ && L \ar[d]^{+} & \\ &&&&}$$

where the row and the column are triangles in
$\mathcal{D}(\mathcal{G})$. Then we have that
$N\cong\text{Ker}_\mathcal{H_\mathbf{t}}(f)$ and
$L\cong\text{Coker}_\mathcal{H_\mathbf{t}}(f)=\varinjlim_\mathcal{H_\mathbf{t}}M_i$.
From the sequence of homologies applied to the column, we get an
exact sequence

\begin{center}
$0\rightarrow H^0(N)\longrightarrow
H^{-1}(Z)\stackrel{g}{\longrightarrow}H^{-1}(L)\rightarrow 0$.
\end{center}
 It then follows that $H^0(N)\cong t(H^{-1}(Z))$ and
$H^{-1}(Z)/t(H^{-1}(Z))\cong H^{-1}(L)$.

We pass to prove the remaining assertions.

\vspace*{0.3cm}

2) Note that, by assertion 1, we have
$H^0(\varinjlim_\mathcal{H_\mathbf{t}}(F_i[1]))=0$. On the other
hand,  when taking $M_i=F_i[1]$ in the last paragraph, the complex
$Z$ can be identified with $\text{cone}(f)[1]$, where
$f:\coprod_{i\leq j}F_{ij}\longrightarrow\coprod_{i\in I}F_i$ is the
colimit-defining morphism. Then we  have $H^{-1}(Z)\cong
H^0(\text{cone}(f))=\text{Coker}(f)\cong\varinjlim F_{i}$, which, by
the last paragraph, implies that
$\varinjlim_\mathcal{H_\mathbf{t}}(F_i[1])\cong (\frac{\varinjlim
F_i}{t(\varinjlim F_i)})[1]$.

\vspace*{0.3cm}

3) a) From assertion 1) we get that $0=\varinjlim
H^{-1}(T_i[0])\longrightarrow
H^{-1}(\varinjlim_{\mathcal{H}_\mathbf{t}}(T_i[0]))$ is an
epimorphism. In particular we have an isomorphism
$\varinjlim_{\mathcal{H}_\mathbf{t}}(T_i[0])\cong
H^0(\varinjlim_{\mathcal{H}_\mathbf{t}}(T_i[0]))[0]$. But, by lemma
\ref{lem.exactness of the Hi}, $H^0$ preserves direct limits and
then the right term of this isomorphism is $(\varinjlim T_i)[0]$.

b) Let us consider the induced triangle $\coprod_{i\leq
j}T_{ij}\stackrel{f}{\longrightarrow}\coprod_{i\in
I}T_i\stackrel{q}{\longrightarrow}Z\stackrel{+}{\longrightarrow}$ in
$\mathcal{D}(\mathcal{G})$. Without loss of generality, we identify
$Z$ with the cone of $f$ in $\mathcal{C}(\mathcal{G})$, so that
$H^{-1}(Z)=\text{Ker}(f)$. But then, by the proved claim that we
made after proving assertion 1 and by the previous paragraph, we get
an isomorphism
$\frac{\text{Ker}(f)}{t(\text{Ker}(f))}\stackrel{\cong}{\longrightarrow}
H^{-1}(\varinjlim_{\mathcal{H}_\mathbf{t}}T_i[0])=0$.
\end{proof}

\begin{lemma} \label{lem.FFTT}
Let us assume that $\mathcal{H}_\mathbf{t}$ is an AB5 category and
let

\begin{center}

$0\rightarrow F_i\longrightarrow F'_i\stackrel{w_i}{\longrightarrow}
T_i\longrightarrow T'_i\rightarrow 0$,
\end{center}
be a direct system of exact sequences in $\mathcal{G}$, with
$F_i,F'_i\in\mathcal{F}$ and $T_i,T'_i\in\mathcal{T}$ for all $i\in
I$. Then the induced morphism $w:=\varinjlim w_i:\varinjlim
F'_i\longrightarrow\varinjlim T_i$ vanishes on $t(\varinjlim F'_i)$.
\end{lemma}
\begin{proof}
Put $K_i:=\text{Im}(w_i)$ and consider the following pullback
diagram and pushout diagram, respectively:

$$\xymatrix{0 \ar[r] & F_{i} \ar[r] \ar@{=}[d] & \tilde{F}_{i} \ar[r] \ar@{^{(}->}[d]  \pushoutcorner & t(K_{i}) \ar[r] \ar@{^{(}->}[d] & 0 \\ 0 \ar[r] & F_{i} \ar[r] & F'_{i} \ar[r] & K_{i} \ar[r] & 0 }$$
$$\xymatrix{0 \ar[r] & K_{i} \ar[r] \ar@{>>}[d] & T_{i} \ar[r] \ar[d]  & T_{i}' \ar[r] \ar@{=}[d] & 0 \\ 0 \ar[r] & \frac{K_i}{t(K_i)} \ar[r] & \tilde{T}_{i} \ar[r] \pullbackcorner & T_{i}' \ar[r] & 0 }$$

The top row of the first diagram and the bottom row of the second
one give exact sequences in $\mathcal{H}_\mathbf{t}$

\begin{center}
$0\rightarrow t(K_i)[0]\longrightarrow
F_i[1]\longrightarrow\tilde{F}_i[1]\rightarrow 0$

$0\rightarrow\tilde{T}_i[0]\longrightarrow T'_i[0]\longrightarrow
K_i/t(K_i)[1]\rightarrow 0$
\end{center}
which give rise to direct systems of short exact sequences in
$\mathcal{H}_\mathbf{t}$. Using now
 the AB5 condition of $\mathcal{H}_\mathbf{t}$ and proposition \ref{prop.direct limit under H0 and H1}, we get exact sequences in $\mathcal{H}_\mathbf{t}$:

 \begin{center}
 $0\rightarrow(\varinjlim t(K_i))[0])\longrightarrow\frac{\varinjlim F_i}{t(\varinjlim F_i)}[1]
 \longrightarrow\frac{\varinjlim\tilde{F}_i}{t(\varinjlim\tilde{F}_i)}[1]\rightarrow 0$

 $0\rightarrow (\varinjlim\tilde{T}_i)[0]\longrightarrow (\varinjlim T'_i)[0]
 \longrightarrow\frac{\varinjlim (K_i/t(K_i))}{t(\varinjlim (K_i/t(K_i)))}[1]\rightarrow
 0$,
 \end{center}
 which necessarily come from exact sequences in $\mathcal{G}$:

 \begin{center}
 $0\rightarrow\frac{\varinjlim F_i}{t(\varinjlim F_i)}\longrightarrow
 \frac{\varinjlim\tilde{F}_i}{t(\varinjlim\tilde{F}_i)}\longrightarrow\varinjlim t(K_i)\rightarrow 0$

 $0\rightarrow\frac{\varinjlim (K_i/t(K_i))}{t(\varinjlim (K_i/t(K_i)))}\longrightarrow\varinjlim\tilde{T}_i\longrightarrow\varinjlim T'_i\rightarrow
 0$.
 \end{center}
 These sequences are obviously induced from applying the direct limit functor to the top row of the first
 diagram and the bottom row of the second diagram above, respectively. With the
 obvious abuse of notation of viewing monomorphisms as inclusions,
 we get that $w:=\varinjlim w_i$ vanishes on $t(\varinjlim\tilde{F}_i)$, that $t(\varinjlim F_i)=(\varinjlim F_i)\cap
 t(\varinjlim\tilde{F}_i)$ and that $t(\varinjlim (K_i/t(K_i)))=0$.
 This last condition in turn implies that $\varinjlim t(K_i)=t(\varinjlim
 K_i)$ since we have an exact sequence $0\rightarrow\varinjlim t(K_i)\longrightarrow\varinjlim K_i\longrightarrow\varinjlim (K_i/t(K_i))\rightarrow
 0$.

 We finally prove that $w(t(\varinjlim F'_i))=0$. Note that
 $\text{Ker}(w)=\text{Ker}(p)$, where $p:\varinjlim F'_i\longrightarrow\varinjlim
 K_i$ is the induced map. Consider now the following bicartesian
 square:

$$\xymatrix{\varinjlim{\tilde{F}_i} \ar[r] \ar@{^{(}->}[d]  \pushoutcorner & \varinjlim{t(K_i)} \ar@{^{(}->}[d]^{\lambda}  \\ \varinjlim{F_{i}'} \ar[r]^{p} & \varinjlim{K_i} \pullbackcorner}$$

Due to the fact that $t(\varinjlim K_i)=\varinjlim t(K_i)$, we have
a unique morphism $\alpha : t(\varinjlim
F'_i)\longrightarrow\varinjlim t(K_i)$ such that $\lambda\circ\alpha
=p_{|t(\varinjlim (F'_i))}$. By the universal property of cartesian
squares, we get a  morphism $u:t(\varinjlim
F'_i)\longrightarrow\varinjlim\tilde{F}_i$ such that the composition
$t(\varinjlim
F'_i)\stackrel{u}{\longrightarrow}\varinjlim\tilde{F}_i\hookrightarrow\varinjlim
F'_i$ is the canonical inclusion. It follows that $u$ is a
monomorphism. Viewing $u$ as in inclusion,  we then have that
$t(\varinjlim F'_i)\subseteq t(\varinjlim \tilde{F}_i)$, and we have
 already seen that $w$ (and hence $p$) vanishes on $t(\varinjlim
 \tilde{F}_i)$.
\end{proof}

Last lemma will be fundamental to prove that the  closure of $\mathcal{F}$ under taking direct limits is a necessary condition for $\mathcal{H}_\mathbf{t}$ to be AB5. We now want to know the information that that closure property gives about the existence of a generator in $\mathcal{H}_\mathbf{t}$. This requires a few preliminary lemmas.

\begin{lemma} \label{lem.H-direct limit of system in CG}
Suppose that $\mathcal{F}$ is closed under taking direct limits in
$\mathcal{G}$. Let $(M_i)_{i\in I}$ be a direct system in
$\mathcal{C}(\mathcal{G})$, where $M_i$ is a complex concentrated in
degrees $-1,0$, for all $i\in I$. If $M_i\in\mathcal{H}_\mathbf{t}$,
for all $i\in I$, then the canonical morphism

\begin{center}
$\varinjlim_\mathcal{H_\mathbf{t}}M_i\longrightarrow\varinjlim_{\mathcal{C}(\mathcal{G})}M_i$
\end{center}
is an isomorphism in $\mathcal{D}(\mathcal{G})$.
\end{lemma}
\begin{proof}
Note that $\varinjlim_{\mathcal{C}(\mathcal{G})}M_i$ is a complex of
$\mathcal{H}_{\mathbf{t}}$ and, hence, there is a canonical morphism
$g:\varinjlim_{\mathcal{H}_{\mathbf{t}}}M_i\longrightarrow\varinjlim_{\mathcal{C}(\mathcal{G})}M_i$
as indicated in the statement. We then get a composition of
morphisms in $\mathcal{G}$:

\begin{center}
$\varinjlim H^k(M_i)\longrightarrow
H^k(\varinjlim_\mathcal{H_\mathbf{t}}M_i)\stackrel{H^k(g)}{\longrightarrow}H^k(\varinjlim_{\mathcal{C}(\mathcal{G})}M_i)$,
\end{center}
for each $k\in\mathbb{Z}$. But all complexes involved have homology
concentrated in degrees $-1,0$ and, by exactness of $\varinjlim$ in
$\mathcal{G}$, we know that the last composition of morphisms is an
isomorphism. By proposition \ref{prop.direct limit under H0 and H1},
we know that the first arrow in the composition is an isomorphism,
for $k=0$, and an epimorphism, for $k=-1$. Then both arrows in the
composition are isomorphisms, for all $k\in\mathbb{Z}$, and so $g$
is an isomorphism in $\mathcal{D}(\mathcal{G})$.
\end{proof}

\begin{lemma} \label{lem.epimorphism in H through homologies}
Let $p:M\longrightarrow N$ be a morphism in $\mathcal{H}_\mathbf{t}$
such that $H^0(p)$ and $H^{-1}(p)$ are epimorphisms 
 in $\mathcal{G}$ and $\text{Ker}(H^0(p))$ is in $\mathcal{T}$. Then $p$ is an epimorphism in
$\mathcal{H}_\mathbf{t}$.
\end{lemma}
\begin{proof}
Let us consider the induced triangle $M\stackrel{p}{\longrightarrow}N\longrightarrow W\stackrel{+}{\longrightarrow}$ in $\mathcal{D}(\mathcal{G})$. The long exact sequence of homologies and the hypotheses give that $H^{-2}(W)\cong\text{Ker}(H^{-1}(p))$ is in $\mathcal{F}$, that $H^{-1}(W)\cong\text{Ker}(H^0(p))$ is in $\mathcal{T}$ and that $H^i(W)=0$, for all $i\neq -2,-1$. It follows that $W[-1]\in\mathcal{H}_\mathbf{t}$, so that we get a short exact sequence $0\rightarrow W[-1]\longrightarrow M\stackrel{p}{\longrightarrow}N\rightarrow 0$ in $\mathcal{H}_\mathbf{t}$. 
\end{proof}

\begin{lemma} \label{lem.F=limF implica T=Pres(V)}
Let $\mathbf{t}=(\mathcal{T},\mathcal{F})$ be a torsion pair such
that $\mathcal{F}$ is closed under taking direct limits in
$\mathcal{G}$. Then there is an object $V$ such that
$\mathcal{T}=\text{Pres}(V)$. Moreover, the torsion pairs such that
$\mathcal{F}$ is closed under taking direct limits in $\mathcal{G}$
form a set.
\end{lemma}
\begin{proof}
Let us fix a generator $G$ of $\mathcal{G}$. Then each object of
$\mathcal{G}$ is a directed union of those of its subobjects which
are isomorphic to quotients $G^{(n)}/X$, where $n$ is a natural
number and $X$ is a subobject of $G^{(n)}$. Let now take
$T\in\mathcal{T}$ and express it as a directed union
$T=\bigcup_{i\in I}U_i$, where $U_i\cong\frac{G^{(n_i)}}{X_i}$, for
some integer $n_i>0$ and some subobject $X_i$ of $G^{(n_i)}$. We now
get an $I$-directed system of exact sequences

\begin{center}
$0\rightarrow t(U_i)\hookrightarrow U_i\longrightarrow
U_i/t(U_i)\rightarrow 0$.
\end{center}
Due to the AB5 condition of $\mathcal{G}$, after taking direct
limits,  we get an exact sequence

\begin{center}
$0\rightarrow\varinjlim t(U_i)\longrightarrow
T\longrightarrow\varinjlim\frac{U_i}{t(U_i)}\rightarrow 0$.
\end{center}
It follows that
$\varinjlim\frac{U_i}{t(U_i)}\in\mathcal{T}\cap\mathcal{F}=0$ since
$\mathcal{F}$ is closed under taking direct limits in $\mathcal{F}$.
Then we have $T=\bigcup_{i\in I}t(U_i)$.

The objects $T'\in\mathcal{T}$ which are isomorphic to subobjects of
quotients $G^n/X$ form a skeletally small subcategory. We take a set
$\mathcal{S}$ of its representatives, up to isomorphism, and put
$V=\coprod_{S\in\mathcal{S}}S$. The previous paragraph shows that
each $T\in\mathcal{T}$ is isomorphic to a direct limit of objects in
$\mathcal{S}$, from which we get that
$\mathcal{T}\subseteq\text{Pres}(V)$. The converse inclusion is
clear.

For the final statement, note that the last paragraph shows that the
assignment $\mathbf{t}\rightsquigarrow\mathcal{S}$ gives an
injective map from the class of torsion pairs $\mathbf{t}$ such that
$\mathcal{F}=\varinjlim\mathcal{F}$ to the set of subsets of
$\bigcup_{n\in\mathbb{N},X<G^n}\mathcal{S}(G^n/X)$, where
$\mathcal{S}(M)$ denotes the set of subobjects of $M$, for each
object $M$.
\end{proof}

We can now give the desired information on the existence of a generator in $\mathcal{H}_\mathbf{t}$. Recall that a \emph{subquotient} of an object $X$ of $\mathcal{G}$ is  a quotient $Y/Z$, where $Z\subseteq Y$ are subobjects of $X$. Note that these subquotients form a set, for each object $X$ in $\mathcal{G}$ (see \cite[Proposition IV.6.6]{S}).

\begin{prop} \label{prop.limF=F implies generator}
Let $\mathbf{t}=(\mathcal{T},\mathcal{F})$ be a torsion pair in $\mathcal{G}$ such that $\mathcal{F}$ is closed under taking direct limits. Then the heart $\mathcal{H}_\mathbf{t}$ has a generator.  
\end{prop}
\begin{proof}
We fix a generator $G$ of $\mathcal{G}$ and, using lemma \ref{lem.F=limF implica T=Pres(V)}, we fix an object $V$ such that $\text{Pres}(V)=\mathcal{T}$. We consider the class $\mathcal{N}$ consisting of the objects $N\in\mathcal{H}_\mathbf{t}$ such that $H^{-1}(N)$ is isomorphic to a subquotient of $G^{(m)}$ and $H^0(N)\cong V^{(n)}$, for some natural numbers $m$ and $n$. We claim that this class is skeletally small. To see this, consider an object $F\in\mathcal{F}$ which is a subquotient of $G^{(m)}$, for some $m\in\mathbb{N}$, and put $\mathcal{N}_{F,n}=\{N\in\mathcal{N}:$ $H^{-1}(N)\cong F\text{ and }H^0(N)\cong V^{(n)}\}$, for each $n\in\mathbb{N}$. Bearing in mind that the subquotients of each $G^{(m)}$ form a set, it is enough  to prove that we have an injective map  $\Psi:\mathcal{N}_{F,n}/\cong\longrightarrow\text{Ext}_\mathcal{G}^2(V^{(n)},F)$ since the codomain of this map is a set. Indeed, we represent any object of $\mathcal{N}_{F,n}$ as a complex concentrated in degrees $-1,0$. If $\cdots \longrightarrow 0\longrightarrow N^{-1}\stackrel{d}{\longrightarrow}N^0\longrightarrow 0 \longrightarrow \cdots$ is such a complex, then $\Psi (N)$ will be the element of $\text{Ext}_\mathcal{G}^2(V^{(n)},F)$ given by the exact sequence

\begin{center}
$0\rightarrow F\longrightarrow N^{-1}\stackrel{d}{\longrightarrow}N^0\longrightarrow V^{(n)}\rightarrow 0$.
\end{center}
To see that $\Psi$ is well defined, we need to check that if $N\cong N'$ then $\Psi (N)=\Psi (N')$. An isomorphism $f:N \xymatrix{\ar[r]^{\sim}& }N'$ is represented by two quasi-isomorphisms $N\stackrel{s}{\longrightarrow}Y\stackrel{s'}{\longleftarrow}N'$, where $Y$ is also a complex concentrated in degrees $-1$ and $0$ (see, e.g.,  the proof of \cite[Theorem 4.2]{CGM}). Then, there is no loss of generality in assuming that $f$ is a quasi-isomorphism, in which case we have a commutative diagram with exact rows:

$$\xymatrix{0 \ar[r] & F \ar[r] \ar@{=}[d] & N^{-1} \ar[r] \ar[d]^{f^{-1}} & N^{0} \ar[r] \ar[d]^{f^{0}} & V^{(n)} \ar[r] \ar@{=}[d] & 0 \\ 0 \ar[r] & F \ar[r] & N^{' -1} \ar[r] & N^{' 0} \ar[r] & V^{(n)} \ar[r] & 0}$$

Then the upper and lower rows of this diagram represent the same element of $\text{Ext}_\mathcal{G}^2(V^{(n)},F)$ (see \cite[Chapter III]{ML}). That is, we have $\Psi (M)=\Psi (N)$, so that  $\Psi:\mathcal{N}_{F,n}/\cong\longrightarrow\text{Ext}_\mathcal{G}^2(V^{(n)},F)$ is well-defined. The injectivity of $\Psi$ follows from   the definition of $\text{Ext}_\mathcal{G}^2(V^{(n)},F)$ via congruencies (see \cite[Chapter III, Section 5]{ML}) and the fact that if we have a commutative diagram as the last one, then the induced chain map $f:N\longrightarrow N'$ is a quasi-isomorphism and, hence, an isomorphism in $\mathcal{H}_\mathbf{t}$. 

 We shall prove that any object of $\mathcal{H}_\mathbf{t}$ is an epimorphic image of a coproduct of objects of $\mathcal{N}$, which will end the proof. Let $M\in\mathcal{H}_\mathbf{t}$ be any object, which we represent by a complex $\cdots \longrightarrow 0\longrightarrow M^{-1}\stackrel{d}{\longrightarrow}M^0\longrightarrow 0 \longrightarrow \cdots$, concentrated in degrees $-1,0$. We fix an epimorphism $p:G^{(J)}\twoheadrightarrow M^{-1}$ in $\mathcal{G}$, for some set $J$. Then, for each finite subset $F\subseteq J$, we have the following commutative diagram with exact rows, where the bottom and pre-bottom left and the upper right squares are cartesian:

$$\xymatrix{0 \ar[r] & U_F \ar@{=}[d] \ar[r] & G^{(F)} \ar@{=}[d] \ar[r] & M^{0}_{F} \ar[r] \ar@{^(->}[d] \pushoutcorner & t(X_F) \ar@{^(->}[d] \ar[r] & 0 \\ 0 \ar[r] & U_F \ar[d] \ar[r] \pushoutcorner & G^{(F)} \ar[r] \ar@{^(->}[d]& M^{0} \ar@{=}[d] \ar[r] & X_F \ar[r] \ar[d] & 0 \\ 0 \ar[r] & U_{J} \ar[r] \ar[d] \pushoutcorner & G^{(J)} \ar@{>>}[d] \ar[r] & M^{0} \ar@{=}[d] \ar[r] & H^{0}(M) \ar[r] \ar@{=}[d] &0&\\ 0 \ar[r] & H^{-1}(M) \ar[r] & M^{-1} \ar[r] & M^{0} \ar[r] & H^{0}(M) \ar[r] & 0}$$
 
The AB5 condition of $\mathcal{G}$ implies that $H^0(M)=\varinjlim (X_F)$, and then the fact that $\mathcal{F}$ is closed under taking direct limits implies that $H^0(M)=\varinjlim (t(X_F))$. This in turn implies that $M^0=\varinjlim (M_F^0)$. We denote by $K_J$ the complex $\cdots \longrightarrow 0\longrightarrow G^{(J)}\stackrel{d\circ p}{\longrightarrow}M^0\longrightarrow 0 \longrightarrow \cdots $, concentrated in degrees $-1,0$. Similarly, for each finite subset $F\subseteq J$, we denote by $K_F$ the complex $\cdots \longrightarrow 0\longrightarrow G^{(F)}\longrightarrow M_F^0\longrightarrow 0 \longrightarrow \cdots$ given by the upper row of last diagram. Note that $K_J$ and $K_F$ are in $\mathcal{U}_\mathbf{t}$. By explicit construction of the functor $L:\mathcal{U}_\mathbf{t}\longrightarrow\mathcal{H}_\mathbf{t}$ (see lemma \ref{lem.adjoints to the inclusions from heart}), we  see that $L(K_J)$ and $L(K_F)$ can be represented by the complexes obtained from $K_J$ and $K_F$ by replacing $G^{(J)}$ and $G^{(F)}$ by $G^{(J)}/t(U_J)$ and $G^{(F)}/t(U_F)$, respectively, in degree $-1$.  This allows us to interpret $(L(X_F))_{F\subseteq J\text{, }F finite}$ as a direct system in $\mathcal{C}(\mathcal{G})$ of complexes concentrated in degrees $-1,0$. But the fact that $\mathcal{F}$ is closed under taking direct limits in $\mathcal{G}$ and $\varinjlim (U_F)=U_J$ implies that $t(U_J)=\varinjlim t(U_F)$, from which we easily get that $L(K_J)=\varinjlim_{\mathcal{C}(\mathcal{G})}(L(K_F))$.  From lemma \ref{lem.H-direct limit of system in CG} we deduce that $L(K_J)\cong\varinjlim_{\mathcal{H}_\mathbf{t}}(L(K_F))$, so that $L(K_J)$ is a quotient in $\mathcal{H}_\mathbf{t}$ of a coproduct of complexes $N'\in\mathcal{H}_\mathbf{t}$ such that $H^{-1}(N')$ is a subquotient of $G^m$, for some $m\in\mathbb{N}$. 

Note that the chain map $K_J\longrightarrow M$ obtained from the diagram above induces a chain map $q:L(K_J)\longrightarrow M$, because $p(t(U_J))\subseteq t(H^{-1}(M))=0$. Moreover, by construction, we have that $H^{-1}(q)$ is an epimorphism and $H^{0}(q)=1_{H^0(M)}$ is an isomorphism. By lemma \ref{lem.epimorphism in H through homologies}, we get that $q$ is an epimorphism in $\mathcal{H}_\mathbf{t}$.
This, together with the previous paragraph, shows that each $M$ in $\mathcal{H}_\mathbf{t}$ is  a quotient of a coproduct of objects $N'$ of $\mathcal{H}_\mathbf{t}$ as above. Replacing $M$ by one such $N'$, we can and shall  restrict ourselves to the case when $M$ is a complex concentrated in degrees $-1,0$, which is in $\mathcal{H}_\mathbf{t}$ and satisfies that $H^{-1}(M)$ is a subquotient of $G^m$, for some $m\in\mathbb{N}$. 

Suppose that $M$ is such a complex in the rest of the proof. By fixing an epimorphism $v:V^{(S)}\twoheadrightarrow H^0(M)$ such that $\text{Ker}(v)\in\mathcal{T}$ and pulling it back along the canonical epimorphism $M^0\twoheadrightarrow H^0(M)$, we obtain a complex $\hat{M}: ...0\longrightarrow M^{-1}\longrightarrow\hat{M}^0\longrightarrow 0...$, concentrated in degrees $-1,0$, such that $H^{-1}(\hat{M})=H^{-1}(M)$, $H^0(\hat{M})=V^{(S)}$ and it comes  with an induced exact sequence  $0\rightarrow \text{Ker}(v)[0]\longrightarrow\hat{M}\stackrel{\hat{v}}{\longrightarrow}M\rightarrow 0$ in $\mathcal{C}(\mathcal{G})$. Since the three terms of this sequence are in $\mathcal{H}_\mathbf{t}$, we get that $\hat{v}$ is an epimorphism in $\mathcal{H}_\mathbf{t}$. Note that $v$ exists because $\mathcal{T}=\text{Pres}(V)$.

Replacing now $M$ by $\hat{M}$, 
we can assume without loss of generality that $H^0(M)=V^{(S)}$, for some set $S$.  Now for each finite subset $F\subset S$, we consider the following commutative diagram with exact rows, whose right square is cartesian. 

$$\xymatrix{0 \ar[r] & H^{-1}(M) \ar[r] \ar@{=}[d] & M^{-1} \ar[r] \ar@{=}[d] & N^{0}_{F} \ar@{^(->}[d] \pushoutcorner \ar[r] & V^{(F)} \ar[r] \ar@{^(->}[d] & 0 \\ 0 \ar[r] & H^{-1}(M) \ar[r] & M^{-1} \ar[r] & M^{0} \ar[r] & V^{(S)} \ar[r] & 0}$$

We denote by $N_F$ the complex $ \cdots \longrightarrow 0\longrightarrow M^{-1}\longrightarrow N_F^0\longrightarrow 0 \longrightarrow \cdots$, concentrated in degrees $-1,0$, given by the upper row of the last diagram. Note that $N_F\in\mathcal{N}$, for each $F\subset S$ finite. Moreover,  $(N_F)_{F\subset S\text{, }F\text{ finite}}$ is a direct system in $\mathcal{C}(\mathcal{G})$ such that $\varinjlim_{\mathcal{C}(\mathcal{G})}(N_F)=M$. By lemma \ref{lem.H-direct limit of system in CG}, we get that $\varinjlim_{\mathcal{H}_\mathbf{t}}(N_F)=M$, so that $M$ is a quotient in $\mathcal{H}_\mathbf{t}$ of a coproduct of objects of $\mathcal{N}$. 
\end{proof}

We are now ready to give the two main results of the paper.

\begin{teor} \label{teor.AB5 heart of a torsion pair}
Let $\mathcal{G}$ be a Grothendieck category, let
$\mathbf{t}=(\mathcal{T},\mathcal{F})$ be a torsion pair in
$\mathcal{G}$, let
$(\mathcal{U}_\mathbf{t},\mathcal{U}_\mathbf{t}^\perp [1])$ be its
associated t-structure in $\mathcal{D}(\mathcal{G})$ and let
$\mathcal{H}_\mathbf{t}=\mathcal{U}_\mathbf{t}\cap\mathcal{U}_\mathbf{t}^{\perp}[1]$
be the heart. The following assertions are equivalent:

\begin{enumerate}
\item[0.] $\mathcal{H}_\mathbf{t}$ is a Grothendieck category.
\item $\mathcal{H}_\mathbf{t}$ is an AB5 abelian category.
\item If $(M_i)_{i\in I}$ is a direct system in
$\mathcal{H}_\mathbf{t}$, with

\begin{center}
$\coprod_{i\leq j}M_{ij}\stackrel{f}{\longrightarrow}\coprod_{i\in
I}M_i\longrightarrow Z\stackrel{+}{\longrightarrow}$
\end{center}
the triangle in $\mathcal{D}(\mathcal{G})$ afforded by the
associated colimit-defining morphism, then the composition
$\varinjlim H^{-1}(M_i)\longrightarrow
H^{-1}(Z)\stackrel{can}{\twoheadrightarrow}H^{-1}(Z)/t(H^{-1}(Z))$
is a monomorphism.

\item For each direct system $(M_i)_{i\in I}$ in $\mathcal{H}_\mathbf{t}$, the
canonical morphism $\varinjlim H^{-1}(M_i)\longrightarrow
H^{-1}(\varinjlim_\mathcal{H_\mathbf{t}}M_i)$ is a monomorphism.
\item For each direct system $(M_i)_{i\in I}$ in $\mathcal{H}_\mathbf{t}$,  the canonical
 morphism $\varinjlim H^k(M_i)\longrightarrow
 H^k(\varinjlim_{\mathcal{H}_\mathbf{t}}(M_i))$ is an isomorphism,
 for all $k\in\mathbb{Z}$.
\end{enumerate}
In that case, the class $\mathcal{F}$ is closed under taking direct limits in $\mathcal{G}$. 
\end{teor}
\begin{proof}

$0)\Longrightarrow 1)$ is clear.

 $1)\Longrightarrow 0)$ By proposition \ref{prop.limF=F implies generator}, we just need to prove that $\mathcal{F}$ is closed under taking direct limits in $\mathcal{G}$. Let $(F_i)_{i\in I}$ be any direct system in $\mathcal{F}$. For each
$j\in I$, let us denote by $\gamma_j$ the composition
$F_j\stackrel{\iota_j}{\longrightarrow}\varinjlim
F_i\stackrel{p}{\longrightarrow}\frac{\varinjlim F_i}{t(\varinjlim
F_i)}$, where the morphisms are the obvious ones. Using the
exactness of direct limits in $\mathcal{G}$, it follows that
$(\text{Ker}(\gamma_i))_{i\in I}$ is a direct system in
$\mathcal{F}$ such that $\varinjlim\text{Ker}(\gamma_i)=t(\varinjlim
F_i)$ is in $\mathcal{T}$. Put $T:=t(\varinjlim F_i)$ and, for each
$i\in I$,  consider the canonical map
$u_i:\text{Ker}(\gamma_i)\longrightarrow T$ into the direct limit.
We then get a direct system of exact sequences in $\mathcal{G}$

\begin{center}
$0\rightarrow\text{Ker}(u_i)\longrightarrow\text{Ker}(\gamma_i)\stackrel{u_i}{\longrightarrow}T\longrightarrow\text{Coker}(u_i)\rightarrow
0$.
\end{center}
From lemma \ref{lem.FFTT} we then get  that the map $u:=\varinjlim
u_i:\varinjlim\text{Ker}(\gamma_i)\longrightarrow T$ vanishes on
$t(\varinjlim\text{Ker}(\gamma_i))$. This implies that $u=0$ since
$\varinjlim\text{Ker}(\gamma_i)=T$ is in $\mathcal{T}$. But $u$ is
an isomorphism by definition of the direct limit. It then follows
that $T=0$, so that $\varinjlim F_i\in\mathcal{F}$ as desired.

$1)\Longrightarrow 2)$ Note that, by the proof of $1)\Longrightarrow 0)$, we know that $\mathcal{F}$ is closed under taking direct limits in $\mathcal{G}$. In particular, if $(M_i)_{i\in I}$ be a direct system in
$\mathcal{H}_\mathbf{t}$ then $t(\varinjlim H^{-1}(M_i))=0$.  For such a direct system, we get an induced direct system of
short exact sequences in $\mathcal{H}_\mathbf{t}$

\begin{center}
$0\rightarrow H^{-1}(M_i)[1]\longrightarrow M_i\longrightarrow
H^0(M_i)[0]\rightarrow 0$.
\end{center}
From the  AB5 condition of $\mathcal{H}_\mathbf{t}$ and proposition
\ref{prop.direct limit under H0 and H1} we  get an exact sequence in
$\mathcal{H}_\mathbf{t}$

\begin{center}
$0\rightarrow \varinjlim H^{-1}(M_i)[1]\longrightarrow\varinjlim_\mathcal{H_\mathbf{t}}M_i\longrightarrow(\varinjlim
H^0(M_i))[0]\rightarrow 0$.
\end{center}
By taking homologies, we get that the canonical morphism 
$\varinjlim H^{-1}(M_i)\longrightarrow
H^{-1}(\varinjlim_\mathcal{H_\mathbf{t}}M_i)$ is a monomorphism in $\mathcal{G}$ and, by the proof of proposition \ref{prop.direct limit under H0 and H1}, we know that $H^{-1}(\varinjlim_\mathcal{H_\mathbf{t}}M_i)\cong H^{-1}(Z)/t(H^{-1}(Z))$.

$2)\Longleftrightarrow 3)\Longleftrightarrow 4)$ follow directly from
proposition \ref{prop.direct limit under H0 and H1}, its proof and
the fact that all complexes in $\mathcal{H}_\mathbf{t}$ have
homology concentrated in degrees $-1$ and $0$.

$4)\Longrightarrow 1)$
$H^0:\mathcal{D}(\mathcal{G})\longrightarrow\mathcal{G}$ is a
cohomological functor, and then $(H^0,1)$ is a cohomological datum
for $\mathcal{H}_\mathbf{t}$. Then the implication  follows from
proposition \ref{prop.exactness of I-colimits}.
\end{proof}

From last theorem we get that the Grothendieck condition of the heart implies  the closure of $\mathcal{F}$ under taking direct limits. One can naturally asks if the converse is also true. Our second main result in the paper shows that this is the case for some familiar torsion pairs. 

\begin{teor} \label{teor.Grothendieck heart2}
Let $\mathcal{G}$ be a Grothendieck category and let
$\mathbf{t}=(\mathcal{T},\mathcal{F})$ be a torsion pair in
$\mathcal{G}$ satisfying, at least,  one the following  conditions:

\begin{enumerate}
\item[a)] $\mathbf{t}$ is  hereditary.
\item[b)] Each object of $\mathcal{H}_\mathbf{t}$ is
isomorphic in $\mathcal{D}(\mathcal{G})$ to a complex $F^\cdot$ such
that $F^k=0$, for $k\neq -1,0$, and $F^k\in\mathcal{F}$, for
$k=-1,0$.
\item[c)] Each object of $\mathcal{H}_\mathbf{t}$ is
isomorphic in $\mathcal{D}(\mathcal{G})$ to a complex $T^\cdot$ such
that $T^k=0$, for $k\neq -1,0$, and $T^k\in\mathcal{T}$, for
$k=-1,0$.

\end{enumerate}

The following assertions are equivalent:

\begin{enumerate}
\item The heart $\mathcal{H}_\mathbf{t}$ is a Grothendieck category;
\item $\mathcal{F}$ is closed under taking direct limits in
$\mathcal{G}$.
\end{enumerate}
\end{teor}
\begin{proof}
$1)\Longrightarrow 2)$ follows
from theorem \ref{teor.AB5 heart of a torsion pair}.

$2)\Longrightarrow 1)$ Let $\mathbf{t}$ be hereditary in this first
paragraph. Let $(M_i)_{i\in I}$ be a direct system in
$\mathcal{H}_\mathbf{t}$. With the terminology theorem \ref{teor.AB5
heart of a torsion pair}, note that the first arrow of the
composition $\varinjlim H^{-1}(M_i)\longrightarrow
H^{-1}(Z)\stackrel{can}{\longrightarrow}H^{-1}(Z)/t(H^{-1}(Z))$ is
always a monomorphism. As a consequence, when $\mathbf{t}$ is
hereditary, the composition is automatically a monomorphism since
its kernel is in $\mathcal{T}\cap\mathcal{F}=0$. By the mentioned
theorem, we get that $\mathcal{H}_\mathbf{t}$ is a Grothendieck category.

 Suppose next that condition b holds.  We
claim that, in that case,  each object of $\mathcal{H}_\mathbf{t}$
is isomorphic to a subobject of an object in $\mathcal{F}[1]$.
Indeed, if $M$ is isomorphic to the mentioned complex $F^\cdot$,
then its differential $d:F^{-1}\longrightarrow F^0$ gives a triangle

\begin{center}
$F^{-1}[0]\longrightarrow F^0[0]\longrightarrow
M\stackrel{+}{\longrightarrow}$
\end{center}
in $\mathcal{D}(\mathcal{G})$ and, hence, it also gives an exact
sequence in $\mathcal{H}_\mathbf{t}$

\begin{center}
$0\rightarrow M\longrightarrow F^{-1}[1]\longrightarrow
F^0[1]\rightarrow 0$.
\end{center}

We want to check that assertion 3 of theorem \ref{teor.AB5 heart of
a torsion pair} holds, for which we will use the fact that
$\mathcal{F}[1]$ is closed under taking quotients in
$\mathcal{H}_\mathbf{t}$. By traditional arguments (see, e.g.
\cite[Corollary 1.7]{AR}), it is not restrictive to assume that the
directed set $I$ is an ordinal and that the given direct system in
$\mathcal{H}_\mathbf{t}$ is continuous (smooth in the terminology of
\cite{AR}). So we start with a direct system $(M_\alpha )_{\alpha
<\lambda}$ in $\mathcal{H}_\mathbf{t}$, where $\lambda$ is a limit
ordinal and $M_\beta =\varinjlim_{\alpha <\beta}M_\alpha$, whenever
$\beta$ is a limit ordinal such that $\beta <\lambda$. Now, by
transfinite induction, we can define a $\lambda$-direct system of
short exact sequences in $\mathcal{H}_\mathbf{t}$

\begin{center}
$0\rightarrow M_\alpha\longrightarrow F_\alpha [1]\longrightarrow
F'_\alpha [1]\rightarrow 0$,
\end{center}
with $F_\alpha ,F'_\alpha\in\mathcal{F}$ for all $\alpha <\lambda$.
Suppose that $\beta =\alpha +1$ is nonlimit and that the sequence
has been defined for $\alpha$. Then the sequence for $\beta$ is the
bottom one of the following commutative diagram, where the upper
left and lower right squares are bicartesian and   F and F' denote
objects of $\mathcal{F}$, and $u_{\alpha + 1}$ any monomorphism into
an object of $\mathcal{F}[1]$:

$$\xymatrix{0 \ar[r] & M_\alpha \pushoutcorner \ar[r] \ar[d]& F_{\alpha}[1] \ar[r] \ar[d] & F_{\alpha}^{'}[1] \ar[r] \ar@{=}[d] & 0 \\ 0 \ar[r] & M_{\alpha + 1} \ar[r] \ar@{=}[d] & N_{\alpha} \pushoutcorner \pullbackcorner \ar[r] \ar[d]^{u_{\alpha+1}} & F_{\alpha}^{'}[1] \ar[r] \ar[d] & 0 \\ 0 \ar[r] & M_{\alpha +1} \ar[r] & F_{\alpha +1}[1] \ar[r] & F^{'}_{\alpha +1}[1] \pullbackcorner \ar[r] & 0 }$$

Suppose now that $\beta$ is a limit ordinal and that the sequence
has been defined for all ordinals $\alpha <\beta$. Using proposition
\ref{prop.direct limit under H0 and H1} and the fact that
$\mathcal{F}$ is closed under taking direct limits, we then get an
exact sequence in $\mathcal{H}_\mathbf{t}$

\begin{center}
$\varinjlim_{\alpha <\beta}M_\alpha\stackrel{g}{\longrightarrow}
(\varinjlim_{\alpha <\beta}F_\alpha)[1]\longrightarrow
(\varinjlim_{\alpha <\beta}F'_\alpha)[1]\rightarrow 0$.
\end{center}
Recall that, by the continuity of the direct system, we have
$M_\beta =\varinjlim_{\alpha <\beta}M_\alpha$. We denote by $W$ the
image of $g$ in $\mathcal{H}_\mathbf{t}$. We then get the following
commutative diagram with exact rows.

$$\xymatrix{0 \ar[r] & \varinjlim_{\alpha <\beta}H^{-1}(M_\alpha ) \ar[r] \ar[d] & \varinjlim{F_{\alpha}} \ar[r] \ar[d]^{\wr} & \varinjlim{F^{'}_{\alpha}} \ar[r] \ar[d]^{\wr} & \varinjlim{H^{0}(M_{\alpha})} \ar[r] \ar[d] & 0 \\ 0 \ar[r] & H^{-1}(W) \ar[r] & H^{-1}(\varinjlim_{\mathcal{H}_{\mathbf{t}}}{F_{\alpha}[1]}) \ar[r] & H^{-1}(\varinjlim_{\mathcal{H}_{\mathbf{t}}}{F^{'}_{\alpha}[1]}) \ar[r] & H^{0}(W) \ar[r] & 0}$$

All the vertical arrows are then isomorphisms since so are the two
central ones.  But the left vertical arrow is the composition
$\varinjlim_{\alpha <\beta}H^{-1}(M_\alpha )\longrightarrow
H^{-1}(\varinjlim_{\alpha <\beta}M_\alpha
)\stackrel{H^{-1}(p)}{\longrightarrow}H^{-1}(W)$, where
$p:\varinjlim_{\alpha <\beta}M_\alpha\longrightarrow W$ is the
obvious epimorphism in $\mathcal{H}_\mathbf{t}$. It follows that the
canonical map $\varinjlim_{\alpha <\beta}H^{-1}(M_\alpha
)\longrightarrow H^{-1}(\varinjlim_{\alpha <\beta}M_\alpha )$ is a
monomorphism. Then it is an isomorphism due to proposition
\ref{prop.direct limit under H0 and H1}(1) and, by this same
proposition and the isomorphic condition of the right vertical arrow
in the above diagram, we conclude that $H^k(p)$ is an isomorphism,
for all $k\in\mathbf{Z}$. Then $p$ is an isomorphism in
$\mathcal{H}_\mathbf{t}$ and the desired short exact sequence for
$\beta$ is defined.

The argument of the previous paragraph, when applied to $\lambda$
instead of $\beta$, shows that the induced morphism
$\varinjlim_{\alpha <\lambda}H^{-1}(M_\alpha )\longrightarrow
H^{-1}(\varinjlim_{\alpha <\lambda}M_\lambda )$ is a monomorphism,
and then assertion 3 of theorem \ref{teor.AB5 heart of a torsion
pair} holds. It follows that $\mathcal{H}_\mathbf{t}$ is a
Grothendieck category.

Finally, suppose that condition c holds.  An argument dual to the
one used for condition b, shows that, $\mathcal{T}[0]$ generates
$\mathcal{H}_\mathbf{t}$. On the other hand, by lemma
\ref{lem.F=limF implica T=Pres(V)}, we have an object
$V\in\mathcal{T}$ such that $\mathcal{T}=\text{Pres}(V)$.  We easily
derive that, for each $T\in\mathcal{T}$, the kernel of the canonical
epimorphism $V^{(Hom_\mathcal{G}(V,T))}\twoheadrightarrow T$ is in
$\mathcal{T}$, so that the induced morphism
$V[0]^{(\text{Hom}_\mathcal{G}(V,T))}\longrightarrow T[0]$ is an
epimorphism in $\mathcal{H}_\mathbf{t}$. Therefore $V[0]$ is a
generator of $\mathcal{H}_\mathbf{t}$.

Note that the assignment $M\rightsquigarrow
\Psi(M):=V[0]^{(\text{Hom}_\mathcal{H_{\mathbf{t}}}(V[0],M))}$ is
functorial. Indeed if $f:M\longrightarrow N$ is a morphism in
$\mathcal{H}_\mathbf{t}$, we define $\Psi
(f):V[0]^{(\text{Hom}_\mathcal{H_{\mathbf{t}}}(V[0],M))}\longrightarrow
V[0]^{(\text{Hom}_\mathcal{H_{\mathbf{t}}}(V[0],N))}$ using the
universal property of the coproduct in $\mathcal{H}_\mathbf{t}$. By
definition,  $\Psi (f)$ the unique morphism in
$\mathcal{H}_\mathbf{t}$ such that $\Psi (f)\circ\iota^M_\alpha
=\iota_{f\circ\alpha}^N$, where $\iota_\alpha^M:V[0]\longrightarrow
V[0]^{(\text{Hom}_\mathcal{H_{\mathbf{t}}}(V[0],M))}$ is the
$\alpha$-injection into the coproduct, where
$\alpha\in\text{Hom}_\mathcal{H_{\mathbf{t}}}(V[0],M)$, and
similarly for $\iota_{f\circ\alpha}^N$.

The functor $\Psi
:\mathcal{H}_\mathbf{t}\longrightarrow\mathcal{H}_\mathbf{t}$ comes
with natural transformation $p:\Psi\twoheadrightarrow
id_\mathcal{H_{\mathbf{t}}}$ which is epimorphic. Note that $\xi
(M):=\text{Ker}(p_M)=T_M[0]$, for some $T_M\in\mathcal{T}$, since
$\mathcal{T}[0]$ is closed under taking subobjects in
$\mathcal{H}_\mathbf{t}$. We then get functors $\Psi ,\xi
:\mathcal{H}_\mathbf{t}\longrightarrow\mathcal{T}\cong\mathcal{T}[0]\hookrightarrow\mathcal{H}_\mathbf{t}$,
together with an exact sequence of functors
$0\rightarrow\xi\stackrel{\mu}{\hookrightarrow}\Psi\stackrel{p}{\twoheadrightarrow}id_\mathcal{H_{\mathbf{t}}}\rightarrow
0$. In particular, if $f:L\longrightarrow M$ are   morphisms in
$\mathcal{H}_\mathbf{t}$, we have a commutative diagram in
$\mathcal{T}\cong\mathcal{T}[0]$

$$\xymatrix{\xi(L) \ar[r]^{\mu_{L}} \ar[d]^{\xi(f)} & \Psi(L) \ar[d]^{\Psi(f)} \\  \xi(M) \ar[r]^{\mu_{M}}  & \Psi(M) }$$

But note that if $M$ is an object  of $\mathcal{H}_\mathbf{t}$,
then, viewing $\xi (M)$ and $\Psi (M)$ as objects of $\mathcal{T}$,
the complex $C(M): \cdots \longrightarrow 0\longrightarrow\xi
(M)\longrightarrow\Psi (M)\longrightarrow 0 \longrightarrow \cdots $
($\Psi (M)$ in degree $0$) is isomorphic to $M$ in
$\mathcal{D}(\mathcal{G})$. The diagram above tells us that the
assignment $M\rightsquigarrow C(M)$ gives a functor
$C:\mathcal{H}_\mathbf{t}\longrightarrow\mathcal{C}(\mathcal{G})$
such that if
$q:\mathcal{C}(\mathcal{G})\longrightarrow\mathcal{D}(\mathcal{G})$
is the canonical functor, then the composition
$\mathcal{H}_\mathbf{t}\stackrel{C}{\longrightarrow}\mathcal{C}(\mathcal{G})\stackrel{q}{\longrightarrow}\mathcal{D}(\mathcal{G})$
is naturally isomorphic to the inclusion
$\mathcal{H}_\mathbf{t}\hookrightarrow\mathcal{D}(\mathcal{G})$.

Suppose that $(M_i)_{i\in I}$ is a direct system in
$\mathcal{H}_\mathbf{t}$. Then $(C(M_i))_{i\in I}$ is a direct
system in $\mathcal{C}(\mathcal{G})$. By lemma \ref{lem.H-direct
limit of system in CG}, we know that
$\varinjlim_{\mathcal{C}(\mathcal{G})}C(M_i)\cong\varinjlim_\mathcal{H_\mathbf{t}}C(M_i)\cong\varinjlim_\mathcal{H_\mathbf{t}}M_i$.
Then we get an isomorphism $\varinjlim
H^{-1}(M_i)=H^{-1}(\varinjlim_{\mathcal{C}(\mathcal{G})}C(M_i))\stackrel{\cong}{\longrightarrow}H^{-1}(\varinjlim_\mathcal{H_\mathbf{t}}M_i)$
(see the proof of lemma \ref{lem.H-direct limit of system in CG}).
Then assertion 3 of theorem \ref{teor.AB5 heart of a torsion pair}
holds.

\begin{cor} \label{cor.generating-cogenerating torsion pair}
Let $\mathbf{t}=(\mathcal{T},\mathcal{F})$ be a torsion pair such
that either $\mathcal{F}$ is generating or $\mathcal{T}$ is
cogenerating. The heart $\mathcal{H}_\mathbf{t}$ is a Grothendieck
category if, and only if, $\mathcal{F}$ is closed under taking
direct limits in $\mathcal{G}$.
\end{cor}
\begin{proof}
We assume that $\mathcal{F}$ is closed under taking direct limits,
because, by theorem \ref{teor.AB5 heart of a torsion pair},  we only
need to prove the 'if' part of the statement.

 Suppose first that $\mathcal{F}$ is a generating class and let
$M\in\mathcal{H}_\mathbf{t}$ be any object, which we represent by a
complex $\cdots \longrightarrow 0\longrightarrow
M^{-1}\stackrel{d}{\longrightarrow}M^0\longrightarrow 0
\longrightarrow \cdots$. By fixing an epimorphism
$p:F^0\twoheadrightarrow M^0$ and taking the pullback of this
morphism along $d$, we may and shall assume that
$M^0=F^0\in\mathcal{F}$. But then $\text{Im}(d)$ is in
$\mathcal{F}$, which implies that $M^{-1}\in\mathcal{F}$ since we
have an exact sequence $0\rightarrow H^{-1}(M)\hookrightarrow
M^{-1}\stackrel{\bar{d}}{\longrightarrow}\text{Im}(d)\rightarrow 0$,
where the outer nonzero terms are in $\mathcal{F}$. Then condition b
of theorem \ref{teor.Grothendieck heart2} holds.

 Suppose that $\mathcal{T}$ is a cogenerating class. Then the injective objects of
$\mathcal{G}$ are in $\mathcal{T}$. By an argument dual to the one
followed in the previous paragraph, we  see that each object
$M\in\mathcal{H}_\mathbf{t}$ is isomorphic in
$\mathcal{D}(\mathcal{G})$ to a complex $\cdots \longrightarrow
0\longrightarrow T^{-1}\longrightarrow T^0\longrightarrow 0
\longrightarrow \cdots$, where $T^{-1}$ is injective and
$T^0\in\mathcal{T}$. Then condition c of theorem
\ref{teor.Grothendieck heart2} holds.
\end{proof}

The following is now a natural question that remains open.

\begin{ques}
Let $\mathbf{t}=(\mathcal{T},\mathcal{F})$ be a torsion pair in the
Grothendieck category $\mathcal{G}$ such that $\mathcal{F}$ is
closed under taking direct limits. Is the heart
$\mathcal{H}_\mathbf{t}$ a Grothendieck (equivalently, AB5) category?.
\end{ques}

\section{Tilting and cotilting torsion pairs revisited}

All throughout this section, the letter $\mathcal{G}$ denotes a
Grothendieck category. We refer the reader to section 2 for the
definition of (co)tilting object in an abelian category.

\begin{defi}
Let $\mathcal{A}$ be an AB3 (resp. AB3*) abelian category. Two
$1$-tilting (resp. $1$-cotilting) objects  of $\mathcal{A}$ are said
to be \emph{equivalent} when their associated torsion pairs
coincide.
\end{defi}

\begin{rem}
Recall that idempotents split in any abelian category. As a
consequence, two $1$-tilting objects $V$ and $V'$ are equivalent if,
and only if,  $\text{Add}(V)=\text{Add}(V')$. Similarly, two
$1$-cotilting objects $Q$ and $Q'$ are equivalent if, and only if,
$\text{Prod}(Q)=\text{Prod}(Q')$.
\end{rem}

With some additional hypotheses, one obtains the following more
familiar characterization of 1-tilting objects. The dual result
characterizes 1-cotilting objects in AB4*  categories with an
injective cogenerator.

\begin{prop} \label{prop.classical definition of tilting}
Let $\mathcal{A}$ be an AB4 abelian category with a projective
generator and let $V$ be any object of $\mathcal{A}$. Consider the
following assertions:

\begin{enumerate}
\item $V$ is a 1-tilting object.
\item The following conditions hold:

\begin{enumerate}
\item There exists an exact sequence $0\rightarrow P^{-1}\longrightarrow P^0\longrightarrow V\rightarrow
0$ in $\mathcal{A}$, where the $P^k$ are projective;
\item $\text{Ext}_\mathcal{A}^1(V,V^{(I)})=0$, for all sets $I$;
\item for some (resp. every) projective generator $P$ of
$\mathcal{A}$, there is an exact sequence

\begin{center}
$0\rightarrow P\longrightarrow V^0\longrightarrow V^1\rightarrow 0$,
\end{center}
where $V^i\in\text{Add}(V)$ for $i=0,1$.
\end{enumerate}
\end{enumerate}
The implications $2)\Longrightarrow 1)$,  $1)\Longrightarrow 2.b$
and $1)\Longrightarrow 2.c$ hold. When $\mathcal{A}$ has enough
injectives, assertions 1 and 2 are equivalent.
\end{prop}
\begin{proof}
As in module categories (see \cite[Proposition 1.3]{CT}, and also
\cite[Section 2]{C}).
\end{proof}

Recall that an object $X\in\text{Ob}(\mathcal{G})$ is called
\emph{self-small} when the canonical morphism
$\text{Hom}_\mathcal{G}(X,X)^{(I)}\longrightarrow\text{Hom}_\mathcal{G}(X,X^{(I)})$
is an isomorphism, for each set $I$. A consequence of the results in
the previous section is the following:

\begin{prop} \label{prop.Grothendieck tilting hearts}
Let $\mathcal{G}$ be a Grothendieck category and let
$\mathbf{t}=(\mathcal{T},\mathcal{F})$ be a torsion pair in
$\mathcal{G}$. Consider the following assertions:

\begin{enumerate}
\item $\mathbf{t}$ is a tilting torsion pair induced by a self-small
1-tilting object;
\item $\mathcal{T}$ is a cogenerating class and the heart $\mathcal{H}_\mathbf{t}$ is a module category;
\item $\mathcal{T}$ is a cogenerating class and $\mathcal{H}_\mathbf{t}$ is a Grothendieck category with a projective generator;
\item  $\mathbf{t}$ is a tilting torsion pair such that $\mathcal{H}_\mathbf{t}$ is an AB5 abelian category;
\item $\mathbf{t}$ is a tilting torsion pair such that $\mathcal{F}$ is
closed under taking direct limits in $\mathcal{G}$.
\end{enumerate}

Then the implications $1)\Longleftrightarrow 2)\Longrightarrow
3)\Longleftrightarrow 4)\Longleftrightarrow 5)$ hold.
\end{prop}
\begin{proof}
Let $V$ be any 1-tilting object and
$\mathbf{t}=(\mathcal{T},\mathcal{F})$ its associated torsion pair.
If $F\in\mathcal{F}$ is any object, then its injective envelope
$E(F)$ and its first cosyzygy $\Omega^{-1}(F)=\frac{E(F)}{F}$ are in
$\mathcal{T}:=\text{Gen}(V)=\text{Ker}(\text{Ext}_\mathcal{G}^1(V,?))$.
It follows that
$\text{Ext}_{\mathcal{H}_t}^1(V[0],F[1])=\text{Ext}_\mathcal{G}^2(V,F)\cong\text{Ext}_\mathcal{G}^1(V,\Omega^{-1}(F))=0$.
On the other hand, we have
$\text{Ext}_{\mathcal{H}_\mathbf{t}}^1(V[0],T[0])\cong\text{Ext}_\mathcal{G}^1(V,T)=0$,
for all $T\in\mathcal{T}$. It follows that $V[0]$ is a projective
object of $\mathcal{H}_\mathbf{t}$. On the other hand, by the proof
of theorem \ref{teor.Grothendieck heart2} under its condition c, we
know that $V[0]$ is a generator of $\mathcal{H}_\mathbf{t}$.

$1)\Longrightarrow 2)$ Let the 1-tilting object  $V$ be self-small.
We then have an isomorphism

\begin{center}
$\text{Hom}_{\mathcal{H}_\mathbf{t}}(V[0],V[0])^{(I)}\cong\text{Hom}_\mathcal{G}(V,V)^{(I)}\cong\text{Hom}_\mathcal{G}(V,V^{(I)})
\cong\text{Hom}_{\mathcal{H}_\mathbf{t}}(V[0],V[0]^{(I)})$,
\hspace*{0.5cm} (*)
\end{center}
for each set $I$. That is, $V[0]$ is a self-small object of
$\mathcal{H}_\mathbf{t}$ and, since it is a projective generator, it
is easily seen that $V[0]$ is a compact object of
$\mathcal{H}_\mathbf{t}$. Then $V[0]$ is a progenerator of
$\mathcal{H}_\mathbf{t}$ and $\mathcal{H}_\mathbf{t}$ is a module
category (see \cite[Corollary 3.6.4]{Po}).

$3)\Longrightarrow 4)$ By the proofs of corollary
\ref{cor.generating-cogenerating torsion pair} and theorem
\ref{teor.Grothendieck heart2}, we know that $\mathcal{T}[0]$
generates $\mathcal{H}_\mathbf{t}$. If  $G$ is a projective
generator of $\mathcal{H}_\mathbf{t}$, then it is necessarily of the
form $G=V[0]$, where $V\in\mathcal{T}$. It easily follows from this
that $\mathcal{T}=\text{Pres}(V)=\text{Gen}(V)$ and, hence, that
$\mathcal{F}=\text{Ker}(\text{Hom}_\mathcal{G}(V,?))$.

From the projectivity of $V[0]$ in $\mathcal{H}_\mathbf{t}$ we get
that
$0=\text{Ext}_{\mathcal{H}_\mathbf{t}}^1(V[0],T[0])=\text{Ext}_\mathcal{G}^1(V,T)$,
for each $T\in\mathcal{T}$. Therefore we get that
$\text{Gen}(V)\subseteq\text{Ker}(\text{Ext}_\mathcal{G}^1(V,?))$.
The proof of the converse inclusion is entirely dual to the
corresponding one for cotilting objects, which is done in the
implication $1)\Longrightarrow 3)$ of proposition
\ref{prop.characterization of cotilting pairs}

$2)\Longrightarrow 1)$ By the argument in the implication
$3)\Longrightarrow 4)$, we can assume that $\mathbf{t}$ is a tilting
torsion pair induced by a 1-tilting object $V$ such that $V[0]$ is a
progenerator of $\mathcal{H}_\mathbf{t}$. From the fact that $V[0]$
is compact in $\mathcal{H}_\mathbf{t}$ we derive that the
isomorphism (*) above still holds. Then $V$ is self-small.

The implication $2)\Longrightarrow 3)$ is clear and
$4)\Longrightarrow 5)$ follows from theorem \ref{teor.AB5 heart of a
torsion pair}.

$5)\Longrightarrow 3)$ That
$\mathcal{T}=\text{Ker}(\text{Ext}_\mathcal{G}^1(V,?))$ is
cogenerating is clear since it contains all injective objects. The
fact that $\mathcal{H}_\mathbf{t}$ is a Grothendieck category
follows then from corollary \ref{cor.generating-cogenerating torsion
pair}. Finally, by the first paragraph of this proof, we know that
$V[0]$ is a projective generator of $\mathcal{H}_\mathbf{t}$.
\end{proof}

\begin{exem}
A projective generator  $P$ of  $\mathcal{G}$ is always a 1-tilting
object. However $P$ is self-small if, and only if, it is compact.
Therefore if $\mathcal{G}$ has a projective generator but is not a
module category, then the trivial torsion pair
$\mathbf{t}=(\mathcal{G},0)$ satisfies assertion 5, but not
assertion 2 of last proposition.
\end{exem}

 However, the following is a natural
question whose answer seems to be unknown.

\begin{ques}
Let $R$ be a ring and $V$ be a 1-tilting $R$-module  such that
$\text{Ker}(\text{Hom}_R(V,?))$ is closed under taking direct limits
in $R-\text{Mod}$. Is $V$ equivalent to a self-small  1-tilting
module?. Note that a 1-tilting $R$-module is self-small if, and only
if, it is finitely presented (cf. \cite[Proposition 1.3]{CT}).
\end{ques}

If $I$ is any set,   then the product  functor $\prod
:\mathcal{G}^I=[I,\mathcal{G}]\longrightarrow\mathcal{G}$ is left
exact, but need not be right exact. We shall denote by
$\prod^1:=\prod_{i\in I}^1:\mathcal{G}^I\longrightarrow\mathcal{G}$
its first right derived functor. Given a family $(X_i)_{i\in I}$, we
have that $\prod_{i\in I}^1X_i$ is the cokernel of the canonical
morphism $\prod_{i\in I}E(X_i)\longrightarrow\prod_{i\in
I}\frac{E(X_i)}{X_i}$.

\begin{defi} \label{def.(strong) cotilting object}
 An object $Q$ of the
Grothendieck category $\mathcal{G}$ will be called \emph{strong
1-cotilting}  when it is $1$-cotilting and  $\prod_{i\in I}^1Q$ is
in $\mathcal{F}:=\text{Cogen}(Q)$, for each set $I$. The
corresponding torsion pair is called a \emph{strong cotilting
torsion pair}.
\end{defi}

 Let
$\mathcal{G}$ be  a locally finitely presented Grothendieck category
in  this paragraph. An exact sequence $0\rightarrow
X\stackrel{u}{\longrightarrow}Y\stackrel{p}{\longrightarrow}
Z\rightarrow 0$ is called \emph{pure-exact} when it is kept exact
when applying the functor $\text{Hom}_\mathcal{G}(U,?)$, for every
finitely presented object $U$. An object $E$ of $\mathcal{G}$ is
\emph{pure-injective} when $\text{Hom}_\mathcal{G}(?,E)$ preserves
the exactness of all pure-exact sequences (see, e.g. \cite{CB} or
\cite{Pr} for details).

\begin{lemma} \label{lem.cotilting-strong-pureinjective}
Let $Q$ be a 1-cotilting object of $\mathcal{G}$. The following
assertions hold:

\begin{enumerate}
\item If $\mathcal{G}$ is AB4* then $Q$ is strong 1-cotilting and the class $\mathcal{F}:=\text{Cogen}(Q)$ is generating;
\item If $\mathcal{G}$ is locally finitely presented, then $Q$ is a
pure-injective object and $\mathcal{F}$ is closed under taking
direct limits in $\mathcal{G}$. In particular, the equivalence
classes of 1-cotilting objects form a set.
\item If there exists a  strong 1-cotilting object  $Q'$
which is equivalent to $Q$, then $Q$ is itself strong 1-cotilting.
\end{enumerate}
\end{lemma}
\begin{proof}
1) That $Q$ is strong 1-cotilting is straightforward since $\prod^1$
vanishes when $\mathcal{G}$ is AB4*. In order  to prove that
$\mathcal{F}$ is generating, it is enough to prove that all
injective objects of $\mathcal{G}$ are homomorphic image of objects
in $\mathcal{F}$. Indeed, if that is the case and $U$ is any object
of $\mathcal{G}$, then fixing an epimorphism $p:F\twoheadrightarrow
E(U)$, with $F\in\mathcal{F}$, and pulling it back along the
inclusion $U\hookrightarrow E(U)$, we obtain an epimorphism
$F'\twoheadrightarrow U$, for some $F'\in\mathcal{F}$. But, by the
dual of proposition \ref{prop.classical definition of tilting},  we
get that each injective cogenerator $E$ is an homomorphic image of
an object in $\mathcal{F}=\text{Cogen}(Q)$.

2) We follow Bazzoni's argument (see \cite{B}) and see that it also
works  in our context. First of all, note that an object
 $Y$ of $\mathcal{G}$ is pure-injective if, and only
if, for every set $S$, each morphism $f:Y^{(S)}\longrightarrow Y$
extends to $Y^S$ (cf. \cite[Theorem 1]{CB}, \cite[Theorem 5.4]{Pr}).
Then lemmas 2.1, 2.3 and 2.4, together with corollary 2.2 of
[op.cit] are valid here. We next consider proposition 2.5 in that
paper. For it to work in our situation, we just need to check that
if $\lambda$ is an infinite cardinal and $(A_\beta
)_{\beta\in\lambda^{\aleph_0}}$ is a family of $\lambda^{\aleph_0}$
subsets of $\lambda$ such that $A_\alpha\cap A_\beta$ is finite, for
all $\alpha\neq\beta$, then the images of the compositions
$M^{A_\beta}\hookrightarrow
M^\lambda\stackrel{pr}{\twoheadrightarrow}\frac{M^\lambda}{M^{(\lambda)}}$
form a family $(Y_\beta )_{\beta\in\lambda^{\aleph_0}}$ of
subobjects of $\frac{M^\lambda}{M^{(\lambda)}}$ which have direct
sum. Note that this amounts to prove that, for each
$\beta\in\lambda^{\aleph_0}$, we have $(M^{(\lambda
)}+M^{A_\beta})\cap (M^{(\lambda
)}+\sum_{\gamma\neq\beta}M^{A_\gamma})=M^{(\lambda )}$.  By the
modular law, which is a consequence of the AB5 condition, we need to
prove that $M^{(\lambda )}+[(M^{(\lambda )}+M^{A_\beta})\cap
(\sum_{\gamma\neq\beta}M^{A_\gamma})]=M^{(\lambda )}$. That is, we
need to prove that $[(M^{(\lambda )}+M^{A_\beta})\cap
(\sum_{\gamma\neq\beta}M^{A_\gamma})]\subseteq M^{(\lambda )}$.

For simplicity, call an object $X$ of $\mathcal{G}$ finitely
generated when it is  homomorphic image of a finitely presented one.
Clearly, $\text{Hom}_\mathcal{G}(X,?)$ preserves direct union of
subobjects in that case.  Due to the locally finite presented
condition of $\mathcal{G}$, each object of this category is a
directed union of finitely generated subobjects. Our task reduces to
prove that if $X$ is a finitely generated subobject of
$[(M^{(\lambda )}+M^{A_\beta})\cap
(\sum_{\gamma\neq\beta}M^{A_\gamma})]$, then $X\subseteq M^{(\lambda
)}$. To do that, we denote by $\text{Supp}(X)$ the set of
$\alpha\in\lambda$ such that the composition $X\hookrightarrow
M^\lambda\stackrel{\pi_\alpha}{\longrightarrow}M$ is nonzero, where
$\pi_\alpha :M^\lambda\twoheadrightarrow M$ is the
$\alpha$-projection, for each $\alpha\in\lambda$. Bearing in mind
that $\sum_{\gamma\neq\beta}M^{A_\gamma}=\bigcup_{F}(\sum_{\gamma\in
F}M^{A_\gamma})$, with $F$ varying on the set of finite subsets of
$\lambda\setminus\{\beta\}$, the AB5 condition (see \cite[V.1]{S})
gives: \begin{center} $X=X\cap (\sum_{\gamma\neq\beta}M^{A_\gamma})=X\cap
[\bigcup_{F}(\sum_{\gamma\in F}M^{A_\gamma})]=\bigcup_{F}[X\cap
(\sum_{\gamma\in F}M^{A_\gamma})]$, \end{center}
and the finitely generated condition of $X$ implies that $X=X\cap
(\sum_{\gamma\in F}M^{A_\gamma})\subseteq\sum_{\gamma\in
F}M^{A_\gamma}$, for some $F\subset\lambda\setminus\{\beta\}$
finite. As a consequence, we have
$\text{Supp}(X)\subseteq\bigcup_{\gamma\in F}A_\gamma$.

On the other hand, exactness of direct limits gives that
$M^{(\lambda)}+M^{A_\beta}=\bigcup_{F'\subset\lambda\text{,
}F'\text{ finite}}[M^{(F')}+M^{A_\beta}]$ and, again by the AB5
condition and the finitely generated condition of $X$, we get that
$X\subseteq M^{(F')}+M^{A_\beta}$, for some finite subset
$F'\subseteq\lambda$. This implies that $\text{Supp}(X)\subseteq
F'\cup A_\beta$. Together with the conclusion of the previous
paragraph, we then get that $\text{Supp}(X)\subseteq (F'\cup
A_\beta)\cap (\bigcup_{\gamma\in F}A_\gamma)$, and so
$\text{Supp}(X)$ is a finite set and $X\subseteq M^{(\lambda )}$.

The previous two paragraphs show that proposition 2.5 and corollary
2.6  of \cite{B} go on in our context. To complete Bazzoni's
argument in our situation, it remains  to check the truth of her
lemma 2.7. This amount to prove that if $0\neq
X\subset\frac{M^\lambda}{M^{(\lambda)}}$ is a finitely generated
subobject, then there exists a morphism
$f:\frac{M^\lambda}{M^{(\lambda)}}\longrightarrow Q$ such that
$f(X)\neq 0$. Indeed, take the subobject $\hat{X}$ of $M^\lambda$
such that $X=\frac{\hat{X}}{M^{(\lambda )}}$. Then
$\text{Supp}(\hat{X})$ is an infinite subset of $\lambda$, and this
allows us to fix a subset $A\subseteq\text{Supp}(\hat{X})$ such that
$|A|=\aleph_0$. If now $p:M^\lambda\twoheadrightarrow M^A$ is the
canonical projection, then we get an induced morphism
$\bar{p}:\frac{M^\lambda}{M^{(\lambda)}}\longrightarrow\frac{M^A}{M^{(A)}}$
such that $\bar{p}(X)\neq 0$. Since
$\frac{M^A}{M^{(A)}}\in\text{Cogen}(Q)$ we get a morphism
$h:\frac{M^A}{M^{(A)}}\longrightarrow Q$ such that
$h(\bar{p}(X))\neq 0$. We take $f=h\circ\bar{p}$ and have $f(X)\neq
0$, as desired. Theferore $Q$ is pure-injective.

Finally, if $(F_i)_{i\in I}$ is a direct system in
$\mathcal{F}=\text{Ker}(\text{Ext}_\mathcal{G}^1(?,Q))$ then the
induced sequence $0\rightarrow K\longrightarrow\coprod_{i\in
I}F_i\stackrel{p}{\longrightarrow}\varinjlim F_i\rightarrow 0$ is
pure-exact. The fact that $\mathcal{F}=\varinjlim\mathcal{F}$
follows, as in module categories, by applying to this sequence the
long exact sequence of $\text{Ext}(?,Q)$. Moreover, equivalence
classes of 1-cotilting objects are in bijection with the cotilting
torsion pairs. Then lemma \ref{lem.F=limF implica T=Pres(V)}
applies.

3) For any set $I$, the functor
$\prod^1:[I,\mathcal{G}]\longrightarrow\mathcal{G}$ is additive.
This and the fact that $\mathcal{F}=\text{Cogen}(Q)$ is closed under
direct summands imply that the class of objects $X$ such that
$\prod_I^1(X)\in\mathcal{F}$ is closed under taking direct summands.
This reduces the proof to check that if $Q$ is strong 1-cotilting,
then $Q^J$ is strong 1-cotilting, for every set $J$. To do that, for
such a set $J$,  we consider the following commutative diagram,
where the upper right square is bicartesian and  where the vertical
sequences are split exacts.

$$\xymatrix{0 \ar[r] & Q^{J} \ar[r] \ar@{=}[d] & E(Q^J) \ar[r] \ar@{^{(}->}[d] \pushoutcorner & \frac{E(Q^J)}{Q^J} \ar[r] \ar@{^{(}->}[d]& 0 \\
0 \ar[r] & Q^{J} \ar[r] & E(Q)^J \ar[r] \ar@{>>}[d] &
\pullbackcorner \hspace{0.2 cm} \frac{E(Q)^{J}}{Q^J}  \ar[r]
\ar@{>>}[d] & 0 \\ & & E \ar@{=}[r] & E&}$$

For each set $I$, the product functor $\prod
:[I,\mathcal{G}]\longrightarrow\mathcal{G}$ preserves pullbacks
since it is left exact. It also preserves split short exact
sequences. It follows that the central square of the following
commutative diagram is bicartesian since the cokernels of its two
vertical arrows are isomorphic:

$$\xymatrix{0 \ar[r] & (Q^{J})^{I} \ar[r] \ar@{=}[d] & (E(Q^J))^{I} \ar[r] \ar@{^{(}->}[d] \pushoutcorner & (\frac{E(Q^J)}{Q^J})^{I} \ar[r] \ar@{^{(}->}[d] & N \ar[r] \ar[d]^{u} &  0 \\
0 \ar[r] & (Q^{J})^{I} \ar[r] & (E(Q)^J)^{I} \ar[r] &
\hspace{0.35cm}(\frac{E(Q)^{J}}{Q^J})^{I} \pullbackcorner \ar[r] &
N' \ar[r] & 0 }$$

Then $u$ is an isomorphism, which allows us to put $N'=N$ and
$u=1_N$. Our goal is to prove that $N$ is in $\mathcal{F}$. But the
lower row of the last diagram fits in a new commutative diagram with
exact rows, where the two left vertical arrows are isomorphisms:

$$\xymatrix{0 \ar[r] & (Q^{J})^{I} \ar[r] \ar[d]^{\wr} & (E(Q)^J)^{I} \ar[r] \ar[d]^{\wr} & (\frac{E(Q)^J}{Q^J})^{I} \ar[r] \ar[d] & N \ar[r] \ar[d]^{v} &  0 \\
0 \ar[r] & Q^{J \times I} \ar[r] & E(Q)^{J \times I} \ar[r] &
(\frac{E(Q)}{Q})^{J \times I} \ar[r] & F \ar[r] & 0 }$$

Due to the left exactness of the product functor, the second
vertical arrow from right to left is a monomorphism. It then follows
that $v$ is also a monomorphism. But $F$ is in $\mathcal{F}$,
because $Q$ is strong 1-cotilting. We then get that
$N\in\mathcal{F}$, as desired.

\end{proof}

\begin{prop} \label{prop.characterization of cotilting pairs}
Let $\mathcal{G}$ be a Grothendieck category and let
$\mathbf{t}=(\mathcal{T},\mathcal{F})$ be a  torsion pair in
$\mathcal{G}$ such that $\mathcal{F}$ is a generating class.
Consider the following assertions:

\begin{enumerate}
\item The heart
$\mathcal{H}_\mathbf{t}$ is a Grothendieck category;
\item  $\mathcal{F}$ is closed under taking direct limits in $\mathcal{G}$;
\item $\mathbf{t}$ is a (strong) cotilting torsion pair.
\end{enumerate}

Then the implications $1)\Longleftrightarrow 2)\Longrightarrow 3)$
hold. When $\mathcal{G}$ is locally finitely presented, all
assertions are equivalent.
\end{prop}
\begin{proof}
Let $(F_i)_{i\in I}$ be a family in $\mathcal{F}$. Note that, in
order to calculate the product
$\prod_{\mathcal{D}(\mathcal{G})}F_i[1]$ in
$\mathcal{D}(\mathcal{G})$, we first replace each $F_i$ by an
injective resolution, which we assume to be the minimal one, and
then take products in $\mathcal{C}(\mathcal{G})$. When $\mathcal{G}$
is not AB4*, the resulting complex can have nonzero homology in
degrees $>0$. However $\prod_{\mathcal{D}(\mathcal{G})}F_i[1]$ is
$\mathcal{U}_\mathbf{t}^\perp [1]$ and, using lemma
\ref{lem.adjoints to the inclusions from heart}, we easily see that
$P:=\prod_{\mathcal{H}_\mathbf{t}}F_i[1]=\tau_\mathcal{U}(\prod_{\mathcal{D}(\mathcal{G})}F_i[1])$.
Due to the fact that
$\mathcal{D}^{>0}(\mathcal{G})\subseteq\mathcal{U}_\mathbf{t}^\perp$,
we have a canonical isomorphism $P\cong\tau_\mathcal{U}(\tau^{\leq
0}(\prod_{\mathcal{D}(\mathcal{G})}F_i[1]))$, where $\tau^{\leq 0}$
denotes the left truncation with respect to the canonical
t-structure $(\mathcal{D}^{\leq 0}(G);\mathcal{D}^{\geq
0}(\mathcal{G}))$. But $\tau^{\leq
0}(\prod_{\mathcal{D}(\mathcal{G})}F_i[1])$ is quasi-isomorphic to
the complex

\begin{center}
$\cdots \longrightarrow 0\longrightarrow\prod_{i\in
I}E(F_i)\stackrel{can}{\longrightarrow}\prod_{i\in
I}\frac{E(F_i)}{F_i}\longrightarrow 0 \longrightarrow \cdots$
\end{center}
concentrated in degrees $-1$ and $0$. It follows easily  that
$H^{-1}(P)\cong\prod_{i\in I}F_i$ and $H^0(P)\cong t(\prod_{i\in
I}^1F_i)$.

$1)\Longleftrightarrow 2)$ is a direct consequence of corollary
\ref{cor.generating-cogenerating torsion pair}.

$1),2)\Longrightarrow 3)$  By the proofs of corollary
\ref{cor.generating-cogenerating torsion pair} and theorem
\ref{teor.Grothendieck heart2}, we know that $\mathcal{F}[1]$
cogenerates $\mathcal{H}_\mathbf{t}$. Then any injective cogenerator
of $\mathcal{H}_\mathbf{t}$ is of the form $Q[1]$, for some
$Q\in\mathcal{F}$. Fixing such a $Q$, we get that $Q[1]^S$ is an
injective cogenerator of $\mathcal{H}_\mathbf{t}$, for each set $S$.
This in turn implies that $Q[1]^S\in\mathcal{F}[1]$. By the initial
paragraph of this proof, we  get that $t(\prod_{s\in S}^1Q)=0$ which implies that $Q[1]^S\cong Q^S[1]$. From this we immediately derive  that
$\mathcal{F}=\text{Copres}(Q)=\text{Cogen}(Q)$. 

We fix an object $Q\in\mathcal{F}$ such that $Q[1]$ is an injective cogenerator of $\mathcal{H}_\mathbf{t}$ and pass to prove that $Q$ is 1-cotilting. The equality  $t(\prod_{s\in S}^1Q)=0$ proved above will give that $Q$ is  strong 1-cotilting and the proof of this implication will be finished. First, the injectivity of $Q[1]$ in
$\mathcal{H}_\mathbf{t}$ implies that
$\text{Ext}_{\mathcal{G}}^1(F,Q)\cong\text{Ext}_{\mathcal{H}_\mathbf{t}}^1(F[1],Q[1])=0$.
From this equality we derive that
$\text{Cogen}(Q)=\mathcal{F}\subseteq\text{Ker}(\text{Ext}_{\mathcal{G}}^1(?,Q))$.

Let now $Z$ be any object in
$\text{Ker}(\text{Ext}_\mathcal{G}^1(?,Q))$. The generating
condition of $\mathcal{F}$ gives us an epimorphism
$p:F\twoheadrightarrow Z$, with $F\in\mathcal{F}$. Putting
$F':=\text{Ker}(p)$, we then get the following commutative diagram,
where the upper right square is bicartesian:

$$\xymatrix{0 \ar[r] & F' \ar@{=}[d] \ar[r] & \tilde{F} \pushoutcorner \ar[r] \ar@{^{(}->}[d] & t(Z) \ar[r] \ar@{^{(}->}[d] & 0 \\ 0 \ar[r] & F' \ar[r] & F \ar[r]^{p} \ar@{>>}[d]& Z \pullbackcorner \ar[r] \ar@{>>}[d] & 0 & (*) \\ && \frac{Z}{t(Z)} \ar@{=}[r] & \frac{Z}{t(Z)}}$$

If we apply the long exact sequence of $\text{Ext}(?,Q)$ to the
central row and the central column of the last diagram, we get the
following commutative diagram with exact rows:

$$\xymatrix{0 \ar[r] & \text{Ext}^{2}_{\mathcal{G}}(\frac{Z}{t(Z)}, Q) \ar[r] \ar[d]^{\alpha} & \text{Ext}^{2}_{\mathcal{G}}(F,Q) \ar[r] \ar@{=}[d] & \text{Ext}^{2}(\tilde{F},Q) \ar[d] \\ 0 \ar[r] & \text{Ext}^{2}_{\mathcal{G}}(Z,Q) \ar[r] & \text{Ext}^{2}_{\mathcal{G}}(F,Q) \ar[r] & \text{Ext}^{2}_{\mathcal{G}}(F', Q)}$$

It follows that $\alpha
:\text{Ext}_\mathcal{G}^2(\frac{Z}{t(Z)},Q)\longrightarrow\text{Ext}_\mathcal{G}^2(Z,Q)$
is a monomorphism. If we now apply the long exact sequence of
$\text{Ext}$ to the right column of the diagram ($\ast$) above, we
obtain that the canonical morphism
$\text{Ext}_\mathcal{G}^1(Z,Q)\longrightarrow\text{Ext}_\mathcal{G}^1(t(Z),Q)$
is an epimorphism, which implies that
$\text{Ext}_\mathcal{G}^1(t(Z),Q)=0$ due to the choice of $Z$. It
follows from this that
$\text{Hom}_{\mathcal{H}_\mathbf{t}}(t(Z)[0],Q[1])=0$, which implies
that $t(Z)=0$ since $Q[1]$ is a cogenerator of
$\mathcal{H}_\mathbf{t}$. We then get
$\text{Ker}(\text{Ext}_\mathcal{G}^1(?,Q))\subseteq\mathcal{F}=\text{Cogen}(Q)$
and, hence, this last inclusion is an equality.

$3)\Longrightarrow 2)$ (assuming that $\mathcal{G}$ is locally
finitely presented) follows directly from lemma
\ref{lem.cotilting-strong-pureinjective}(2).
\end{proof}

\begin{rems} \label{rem. cotilting = strong cotilting in locally fp}
\begin{enumerate}
\item  When $\mathcal{G}$ is locally finitely presented, by proposition \ref{prop.characterization of cotilting pairs} and
lemma \ref{lem.cotilting-strong-pureinjective}(3), we know that if
$Q$ is a $1$-cotilting object such that
$\mathcal{F}=\text{Cogen}(Q)$ is a generating class of
$\mathcal{G}$, then $Q$ is strong 1-cotilting.
\item When $\mathcal{G}$ is AB4*, it follows from  lemma \ref{lem.cotilting-strong-pureinjective} and
proposition \ref{prop.characterization of cotilting pairs} that the
following assertions are equivalent for a torsion pair
$\mathbf{t}=(\mathcal{T},\mathcal{F})$:

\begin{enumerate}
\item $\mathcal{F}$ is generating and closed under taking direct
limits in $\mathcal{G}$ \item $\mathbf{t}$
 is a strong cotilting torsion pair
such that $\mathcal{F}$ is closed under taking direct limits in
$\mathcal{G}$.
\end{enumerate}
\end{enumerate}

\end{rems}

The following  direct consequence of  proposition
\ref{prop.characterization of cotilting pairs}  extends
\cite[Corollary 6.3]{CMT}.

\begin{cor} \label{cor.tilting pairs are cotilting}
Let $V$ be a 1-tilting object such that
$\mathcal{F}=\text{Ker}(\text{Hom}_\mathcal{G}(V,?))$ is closed
under taking direct limits in $\mathcal{G}$ (e.g., when $V$ is
self-small). If $\mathcal{F}$ is a generating class, then the
torsion pair
$\mathbf{t}=(\text{Gen}(V),\text{Ker}(\text{Hom}_\mathcal{G}(V,?)))$
is strong cotilting.
\end{cor}

We now  make explicit what proposition \ref{prop.characterization of
cotilting pairs} says in case $\mathcal{G}$ is locally finitely
presented and AB4* (see lemma
\ref{lem.cotilting-strong-pureinjective}):

\begin{cor} \label{cor.final}
Let $\mathcal{G}$ be locally finitely presented and AB4* and let
$\mathbf{t}=(\mathcal{T},\mathcal{F})$ be a torsion pair in
$\mathcal{G}$. The following assertions are equivalent:

\begin{enumerate}
\item $\mathcal{F}$ is a generating class and the heart
$\mathcal{H}_\mathbf{t}$ is a Grothendieck category;
\item $\mathcal{F}$ is a generating class closed under taking direct
limits in $\mathcal{G}$;
\item $\mathbf{t}$ is a cotilting torsion pair.
\end{enumerate}
\end{cor}

\begin{rems}
\begin{enumerate}
\item  The last corollary
 extends the main result of \cite{CG} (see \cite[Theorem 6.2]{M}).
 \item All Grothendieck categories with enough projectives are AB4*,
 but the converse is not true and there even exist AB4* Grothendieck categories with no nonzero projective objects (see \cite[Theorem
 4.1]{Roo}).

\end{enumerate}
 \end{rems}

In the particular case when $\mathcal{G}=R-\text{Mod}$, for a ring
$R$, a torsion pair $\mathbf{t}$ has the property that $\mathcal{F}$
is generating if, and only if, $\mathbf{t}$ is \emph{faithful}. That
is, if and only if, $R\in\mathcal{F}$. In a sense, Bazzoni's result
(see \cite[Theorem 2.8]{B}) states that if $\mathbf{t}$ is a
cotilting torsion pair in $R-\text{Mod}$ then its torsionfree class
is closed under taking direct limits (and $\mathbf{t}$ is faithful).
By corollary \ref{cor.final}, we also have the converse, which, as
Silvana Bazzoni pointed out to us,  can be also deduced from
\cite[Corollary 8.1.10]{GT}:

\begin{cor} \label{cor.Bazzoni's converse}
Let $R$ be a ring. A torsion pair
$\mathbf{t}=(\mathcal{T},\mathcal{F})$ in $R-\text{Mod}$ is
cotilting if, and only if, it is is faithful and $\mathcal{F}$ is
closed under taking direct limits.
\end{cor}

 Recall that a Grothendieck category is called \emph{locally
 noetherian} when it has a set of noetherian generators. The
 following result extends \cite[Theorem A]{BK} (see lemma \ref{lem.cotilting-strong-pureinjective}):

 \begin{cor} \label{cor.Buan-Krause}
Let $\mathcal{G}$ be a locally finitely presented Grothendieck
category which is locally noetherian and denote by
$\text{fp}(\mathcal{G})$ its full subcategory of finitely presented
(=noetherian) objects.   There is a one-to-one correspondence
between:

\begin{enumerate}
\item The torsion pairs $(\mathcal{X},\mathcal{Y})$ of $\text{fp}(\mathcal{G})$ such that $\mathcal{Y}$ contains a set of generators;
\item The equivalence classes of  1-cotilting objects $Q$ of
$\mathcal{G}$ such that $\text{Cogen}(Q)$ is a generating class.

When, in addition, $\mathcal{G}$ is an AB4* category, they are also
in bijection with
\item The equivalence classes of 1-cotilting objects of $\mathcal{G}$.
\end{enumerate}
The map from 1 to 2 takes $(\mathcal{X},\mathcal{Y})$ to the
equivalence class $[Q]$, where $Q$ is a 1-cotilting object such that
$\text{Cogen}(Q)=\{F\in\text{Ob}(\mathcal{G}):$
$\text{Hom}_\mathcal{G}(X,F)=0\text{, for all }X\in\mathcal{X}\}$.
The map from 2 to 1 takes $[Q]$ to
$(\text{Ker}(\text{Hom}_\mathcal{G}(?,Q))\cap\text{fp}(\mathcal{G}),\text{Cogen}(Q)\cap\text{fp}(\mathcal{G}))$.
 \end{cor}
 \begin{proof}
 By lemma \ref{lem.cotilting-strong-pureinjective}, when $\mathcal{G}$ is
 AB4*, the classes in 2) and 3) are the same. We then prove the
 bijection between 1) and 2). Given a torsion pair
 $(\mathcal{X},\mathcal{Y})$ in $\text{fp}(\mathcal{G})$ as in 1), by \cite[Lemma 4.4]{CB}, we know that the torsion pair in
$\mathcal{G}$ generated by $\mathcal{X}$ is
$\mathbf{t}=(\mathcal{T},\mathcal{F})=(\varinjlim\mathcal{X},\varinjlim\mathcal{Y})$.
It follows from proposition \ref{prop.characterization of cotilting
pairs}  that $\mathbf{t}$ is a cotilting torsion pair. We then get a
 1-cotilting object $Q$, uniquely determined up to equivalence,  such that
$\text{Cogen}(Q)=\varinjlim\mathcal{Y}=\{F\in\mathcal{G}:$
$\text{Hom}_\mathcal{G}(X,F)=0,\text{ for all }X\in\mathcal{X}\}$.

Suppose now that $Q$ is any 1-cotilting object and its associated
torsion pair $\mathbf{t}=(\mathcal{T},\mathcal{F})$ has the property
that $\mathcal{F}$ is a generating class.  Then
$(\mathcal{X},\mathcal{Y}):=(\mathcal{T}\cap\text{fp}(\mathcal{G}),\mathcal{F}\cap\text{fp}(\mathcal{G}))$
is a torsion pair in $\text{fp}(\mathcal{G})$. We claim that
$\mathcal{Y}$ contains a set of generators. Indeed, by hypothesis
$\mathcal{F}$ contains a generator $G$ of $\mathcal{G}$. By the
locally noetherian condition of $\mathcal{G}$, we know that $G$ is
the direct union of its noetherian (=finitely presented) subobjects.
Then the finitely presented subobjects of $G$ form a set of
generators of $\mathcal{G}$ which is in $\mathcal{Y}$, thus settling
our claim.

On the other hand,  in the situation of last paragraph, we have that
$(\varinjlim\mathcal{X},\varinjlim\mathcal{Y})$ is a torsion pair in
$\mathcal{G}$ such that $\varinjlim\mathcal{X}\subseteq\mathcal{T}$
and $\varinjlim\mathcal{Y}\subseteq\mathcal{F}$. Then these
inclusions are equalities and, hence,  $\mathbf{t}$ is the image of
$(\mathcal{X},\mathcal{Y})$ by the map from 1 to 2 defined in the
first  paragraph of this proof. That the two maps, from 1 to 2 and
from 2 to 1,  are mutually inverse is then a straightforward
consequence of this.
\end{proof}

 \begin{exems}
 The following are examples of locally finitely presented
 Grothendieck categories. So proposition \ref{prop.characterization of cotilting pairs} and corollary \ref{cor.final} apply to
 them.

 \begin{enumerate}
 \item Each category of additive functors $\mathcal{A}\longrightarrow\text{Ab}$,
 for every skeletally small additive category
 $\mathcal{A}$. Equivalently (see \cite[Proposition II.2]{G}),
 each category $R-\text{Mod}$ of unitary modules over a ring $R$ with enough idempotents.   \item The category $\text{Qcoh}(\mathbf{X})$ of
 quasi-coherent sheaves over any quasi-compact and quasi-separated
 algebraic scheme (\cite[I.6.9.12]{GD}). When, in addition,  $\mathbf{X}$ is  locally Noetherian, corollary \ref{cor.Buan-Krause} also applies to
 $\mathcal{G}=\text{Qcoh}(\mathbf{X})$. \item Each quotient category
 $\mathcal{G}/\mathcal{T}$, where $\mathcal{G}$ is a locally
 finitely presented Grothendieck category and $\mathcal{T}$ is a
 hereditary torsion class in $\mathcal{G}$ generated by finitely
 presented objects (\cite[Proposition 2.4]{ES}).
 \end{enumerate}
 \end{exems}

We end the paper with the following question:

\begin{ques}
Let $\mathcal{G}$ be a locally finitely presented Grothendieck
category
 and $Q$ be a 1-cotilting
object. Is $\mathcal{F}=\text{Cogen}(Q)$ a generating class of
$\mathcal{G}$?.

\end{ques}

\end{proof}

\end{document}